\documentclass[11pt]{article}

\usepackage{subcaption,graphicx}%
\usepackage{multirow}%
\usepackage{amsmath,amssymb,amsfonts}%
\usepackage{amsthm}%
\usepackage{mathrsfs}%
\usepackage{textcomp}%
\usepackage{manyfoot}%
\usepackage{booktabs}%
\usepackage{algorithm}%
\usepackage{algorithmicx}%
\usepackage{algpseudocode}%
\usepackage{listings}%
\usepackage{natbib}

\usepackage{natbib}
\usepackage{dsfont}
\usepackage{bm}
\usepackage{authblk}

\RequirePackage[%
reversemp,
paperwidth=210mm,
paperheight=297mm,
top={26mm},
headheight={12pt},
headsep={5.15mm},
text={160mm,216mm},
marginparsep=5mm,
marginparwidth=12mm,
bindingoffset=6mm,
footskip=10.13mm]{geometry}%

\theoremstyle{plain}%
\newtheorem{thm}{Theorem}
\newtheorem{con}[thm]{Condition}

\theoremstyle{remark}%
\newtheorem{rem}{Remark}%
\newtheorem{lem}{Lemma}

\theoremstyle{definition}%

\newcommand{\RR}{\mathbb{R}}
\newcommand{\NN}{\mathbb{N}}
\newcommand{\ZZ}{\mathbb{Z}}

\newcommand{\EE}{\mathbb{E}}
\newcommand{\PP}{\mathbb{P}}

\newcommand{\one}{\mathds{1}}

\newcommand{\thresh}[1]{\overline{#1}}

\begin{document}
	
\title{Asymptotic behavior of spatio-temporal point processes of exceedances}

\author[1]{Carolin Forster}
\author[1,2]{Marco Oesting}
\affil[1]{Institute for Stochastics and Applications, University of Stuttgart}
\affil[2]{Stuttgart Center for Simulation Science (SC SimTech),  University of Stuttgart}

\date{}

\maketitle

\begin{abstract}
In this paper, we analyze the asymptotic behavior of the point process of exceedances in a spatio-temporal setting whose points are given by the rescaled occurrence times, the sites and the rescaled values of exceedances. Here, the exceedances over a high threshold
are flexibly defined via site-dependent risk functionals.   
Exploiting the framework of stationary regularly varying multivariate time series, we merge and extend the results from the literature in order to show weak convergence of the considered point processes of extremes and to explicitly
determine its limit distribution.
\end{abstract}


\section{Introduction}

Precise statistical modeling of spatial and spatio-temporal extremes is important to gain a better understanding and to analyze the risk of events such as heavy precipitation \citep{cooley-etal-2007,buishand-etal-2008,castro-huser-2020,reich-shaby-2012,hazra-reich-staicu-2020,bopp-etal-2021,forster2022non}, strong wind gusts \citep{oesting-etal-2017,castro-etal-2019}, severe air pollution \citep{eastoe-tawn-2009,vettori-etal-2019}, extreme river discharges \citep{katz-etal-2002, asadi-etal-2015, oesting-schnurr-2020} or temperature extremes \citep{EdFO-2019, castro-etal-2021}. Within this framework, modeling of the occurrence of extreme events in space and time is an important task, which is typically done via point processes, see, for instance, \citet{koh2023spatiotemporal}.

The use of point processes in extreme value theory is usually justified by limit theorems that allow for an extrapolation of point patterns beyond the data. Existing literature mainly studies the limiting behavior of point processes of one of the forms
\begin{align}
  &  \sum_{t=1}^{n}\delta_{a_{n}^{-1}\bm X_{t}}, \quad n \geq 1, \label{N_n_allgemein}\\
  &  \sum_{t=1}^{n}\delta_{(t/n,a_{n}^{-1}\bm X_{t})}, \quad n \geq 1,\label{tilde_N_n_allgemein}
\end{align}
where $\delta_x$ denotes the Dirac measure at some point $x$, $\{\bm X_{t}, \, t \in \mathbb{Z}\}$ denotes the multivariate time series of interest and $a_n$ are appropriate normalizing constants, under various assumptions.
For a time series $\{\bm X_{t}, \, t \in \mathbb{Z}\}$ of independent and identically distributed random vectors, it is well known that the limit point processes of \eqref{N_n_allgemein} and \eqref{tilde_N_n_allgemein} converge to a Poisson point process under mild assumptions on the distribution of $\bm X_0$, see, e.g., Proposition~3.21 in \citet{resnick2008extreme}.

More generally, point processes of exceedances over a high threshold $a_n>0$ without the assumption of independence for the time series have been considered. For a univariate stationary time series $\{X_{t}, \, t \in \mathbb{Z}\}$, \citet{leadbetter1976} and \citet{leadbetter1983} define an exceedance over $a_n$ at time $t$ as $X_{t} > a_n$.   
Based on this, 
they studied the point process of rescaled occurrence times of exceedances, i.e.,
$$\sum_{t=1}^{n}\delta_{t/n}\mathds{1}\{X_{t} > a_{n}\}, \quad n \geq 1,$$
and showed that, under certain mixing and local dependence conditions, this process converges in distribution to a Poisson point process. 
\citet{hsing1988exceedance} explored this point process of rescaled occurrence times of exceedances under more general conditions and derived that every limiting point process of exceedances is a compound Poisson process.

Imposing additional assumptions on $\{\bm X_{t}, \, t \in \mathbb{Z}\}$ such as multivariate regular variation,  \citet{davis1998} and \citet{davis1995point} derive results for the limit point process of \eqref{N_n_allgemein}. Further structural results in the case of multivariate regular variation are obtained by \citet{basrak2009regularly} who introduce the so-called tail process and spectral tail process and by this means discover that some additional conditions are redundant. Based on the tail process and spectral tail process, they derive an explicit representation for the weak limit of \eqref{N_n_allgemein} in terms of Laplace functionals and distributional representations.
Further convergence results for point processes of the form \eqref{N_n_allgemein} and \eqref{tilde_N_n_allgemein} can be found in \citet{kulik2020heavy}. 

While all these results have been shown for general multivariate time series, further quantities of interest appear in a spatio-temporal setting where $\bm X_t = (X_t(s))_{s \in \mathcal{S}}$ for some spatial domain $\mathcal{S}$. In this case, not only the time $t$, but also the site $s$ of the occurrence of an extreme event is of interest. However, limit results for such spatio-temporal point processes are still missing in the literature.  To fill this gap, in this paper, we will assume that $\{\bm X_{t}, \, t \in \mathbb{Z}\}$ is a stationary spatially indexed time series, see Section \ref{Preliminaries} for more details, and consider point processes of exceedances that are based on site-dependent risk functionals and consisting of the rescaled values of the time series, the rescaled occurrence times and sites. For this point processes, we derive an explicit representation for the weak limit with respect to Laplace functionals and distributional representations. Moreover, we give an explicit distributional representation for the corresponding limit marked point processes of exceedances. 

The outline of this paper is as follows: In Section~¸\ref{Preliminaries}, we provide theoretical background on jointly regular variation. Section~\ref{Point process convergence} contains our main results on the asymptotic behavior of the considered point processes of exceedances. The corresponding proofs are postponed to Section~\ref{Proofs}.

\section{Preliminaries}
\label{Preliminaries}
Henceforth, we will always assume that $\{\bm X_{t}, \, t \in \mathbb{Z}\}$ is a stationary non-negative spatially indexed time series, i.e., $\bm X_t = (X_t(s))_{s \in \mathcal{S}}$ takes values in $E = [0,\infty)^{\mathcal{S}}$
for some discrete finite spatial domain $\mathcal{S}$ for all $t \in \mathbb{Z}$ and is stationary in time, and equip the space $E$ with the supremum norm $\|\bm x\| := \max_{s \in \mathcal{S}} x(s)$. 

The main assumption in this paper will be that the stationary time series will be jointly regularly varying. In order to define regular variation and link its definition to the concept of vague convergence, it is convenient to provide sets bounded away from zero with useful topological properties by modifying the space $E$.
To this end, we first consider the extended space $\overline{E} = [0,\infty]^{\mathcal{S}}$ equipped with a bounded metric $d_{\overline{E}}$ that is compatible with the standard product topology, e.g.,
$$ d_{\overline{E}}(\bm x, \bm y) = \max_{s \in S} \left| \frac{x(s)}{1+x(s)} - \frac{y(s)}{1+y(s)}\right|, \quad \bm x = (x(s))_{s \in S}, \ \bm y = (y(s))_{s \in S} \in \overline{E}.  $$
Then, the one-point uncompactification $\overline{E}_0 = [0,\infty]^\mathcal{S} \setminus \{\bf 0\}$ \citep{resnick2008extreme} equipped with the metric
$$ d_{\overline{E}_0}(\bm x, \bm y) = d_{\overline{E}}(\bm x, \bm y) \vee 
  \left| \frac{1}{d_{\overline{E}}(\bm 0, \bm x)} - \frac{1}{d_{\overline{E}}(\bm 0, \bm y)} \right|, \qquad \bm x, \bm y \in \overline{E}_0,
$$
becomes a locally compact complete separable metric space and a set $A \subset \overline{E}_0$ is bounded with respect to $d_{\overline{E}_0}$ (and even relatively compact) if and only if it is separated from $\bm 0 \in \RR^{\mathcal{S}}$ in the sense that
$ \inf\{ \|\bm x\|: \, \bm x \in A\} > 0,$
see \cite{kulik2020heavy}. Analogous results hold for the higher-dimensional space  $\overline{E}^T_0 := ([0,\infty]^{\mathcal{S}})^T \setminus \{\bm 0\}$ where $T \subset \ZZ$ is finite. In particular, a set $A \subset \overline{E}^T_0$ is relatively compact if and only if
$$ \inf\left\{ \max_{t \in T} \|\bm x_t\|: \, \bm x \in A\right\} > 0.$$

With this notion in mind, we follow \cite{basrak2009regularly} and call the time series $\{\bm X_{t},  \,t \in \mathbb{Z}\}$ jointly regularly varying with index $\alpha>0$ if
\begin{itemize}
    \item $ \lim_{x \to \infty} \PP(\|\bm X_0\| > x t \mid \|\bm X_0\| > x) = t^{-\alpha}$ for all $t > 1$ 
    and
    \item for every $t_1 < t_2 \in \ZZ$, the vector $(\bm X_{t_1}, \ldots, \bm X_{t_2})$ is multivariate regularly varying with the same scaling function $x \mapsto \PP(\|\bm X_0\| > x)$, i.e., there exists a non-zero measure $\mu_{\{t_1,\ldots,t_2\}}$
    on $\overline{E}^{\{t_1,\ldots,t_2\}}_0$ such that $\mu_{\{t_1,\ldots,t_2\}}(A)$ is finite for every relatively compact set $A \subset \overline{E}^{\{t_1,\ldots,t_2\}}_0$ and 
    $$ \frac{1}{\PP(\|\bm X_0\| > x)} 
    \PP\left( \left(\frac{\bm X_{t_1}} x, \ldots, \frac{\bm X_{t_2}} x \right) \in A \right) \stackrel{x \to \infty}{\longrightarrow}  \mu_{\{t_1,\ldots,t_2\}}(A) $$
    for every relatively compact set $A \subset \overline{E}^{\{t_1,\ldots,t_2\}}_0$ with $\mu_{\{t_1,\ldots,t_2\}}(\partial A)=0$.
\end{itemize}

Corollary 3.2 in \citet{basrak2009regularly} implies that the joint regular variation of the time series $\{\bm X_{t}, \,t \in \mathbb{Z}\}$ with index 
$\alpha$ is equivalent to the existence of a process
$\{\bm \Theta_{t}, \, t \in \mathbb{Z}\}$ with  $\bm \Theta_{t} = (\Theta_t(s))_{s \in \mathcal{S}}$ such that for all $t_1 < t_2 \in \mathbb{Z}$, we have the weak convergence
\begin{equation*} 
\mathcal{L} \left\{\frac{\bm X_{t_1}}{\| \bm X_0\|}, \hdots, \frac{\bm X_{t_2}}{ \| \bm X_0\|} \, \bigg| \, \| \bm X_0\| >x\right\} \stackrel{x \to \infty}{\longrightarrow} \mathcal{L}\{\bm \Theta_{t_1},\hdots,\bm \Theta_{t_2}\}.
\end{equation*}
The process $\{\bm \Theta_{t}, \, t \in \mathbb{Z}\}$ is called the spectral tail process of $\{\bm X_{t}, \,  t \in \mathbb{Z}\}$ and, by definition, satisfies $\|\bm \Theta_0\|=1$ almost surely. Closely related to the spectral tail process $\{\bm \Theta_{t}, \, t \in \mathbb{Z}\}$ is the tail process $\{\bm Y_{t}, \, t \in \mathbb{Z}\}$ given by $\bm Y_t = P_\alpha \bm \Theta_t$, $t \in \ZZ$, where $P_\alpha$ is an $\alpha$-Pareto distributed random variable independent of $\{\bm \Theta_{t}, \, t \in \mathbb{Z}\}$. Noting that $\bm \Theta_t = \bm Y_t / \|\bm Y_0\|$, we see that there is a one-to-one correspondence between the tail process and the spectral tail process of a jointly regularly varying time series.

\section{Point process convergence}
\label{Point process convergence}

In this paper, we want to study point processes consisting of the occurrence times $ t \in \NN$ and the sites $ s \in \mathcal{S}$ of extremes, as well as the appropriately normalized spatial shape $\bm X_t$ of the event.
Here, the occurrence of an extreme event at $(t,s) \in \NN \times \mathcal{S}$ is defined as the exceedance of a site-dependent risk functional $r^{(s)}(\bm X_t)$ over a high threshold. More precisely, for any $s \in \mathcal{S}$, let $r^{(s)}: E \to [0, \infty)$
be a functional that is continuous at the origin and positively homogeneous, i.e., $r^{(s)}(c \bm x) = c r^{(s)}(\bm x)$ for all $c > 0$ and $\bm x \in E$. Furthermore,
we consider a positive scaling sequence $\{a_n\}$ with 
\begin{equation} \label{a_n_sequence}
n \PP(\|\bm X_0\| > a_n) \to 1
\end{equation}
and define the point processes of space-time occurrences of exceedances 
\begin{equation}
\label{N_n}  
N_n(u) := \sum_{s \in \mathcal{S}}\sum_{j=1}^{n}\delta_{(j/n,s,a_{n}^{-1}\bm X_{j})} \mathds{1}\{r^{(s)}(\bm X_{j}) > a_{n} u\}, \quad n \geq 1,
\end{equation}
on $[0,1] \times \mathcal{X}_{u} \subset [0,1] \times \mathcal{S} \times \overline{E}_0$
where
\begin{equation*}
 \mathcal{X}_{u} := \{(s,\bm x) \in \mathcal{S} \times (E \setminus \{0\}): r^{(s)}(\bm x) > u\}.     
\end{equation*}

Here, without loss of generality, we may assume that the threshold $u$ is chosen sufficiently large such that 
    \begin{equation} \label{eq:risk-bound}
        r^{(s)}(\bm x) \leq u \cdot \|\bm x\|, \qquad \text{for all } \bm x \in E \text{ and } s \in \mathcal{S}. 
    \end{equation}
Note that such a threshold always exists, as the risk functionals are assumed to be continuous at the origin and homogeneous. However, the choice of such a threshold $u$ is not unique. The larger the value of $u$ in \eqref{eq:risk-bound} is chosen, the fewer points are considered as exceedances in the sense of the point process $N_n(u)$.

Note that, so far, most of the literature on extremes considers a single risk functional $r: E \to [0,\infty)$ and defines $\bm X_t$ to be extreme if $r(\bm X_t) > a_n u$. Here, using site-dependent functionals allows for more flexible definitions of extremes. We define $\bm X_t$ as extreme if $r^{(s)}(\bm X_t) > a_n u$ for at least one $s \in \mathcal{S}$. In particular, this includes the following cases:
\begin{itemize}
    \item $r^{(s)}(\bm x) \equiv r(\bm x)$ for some risk functional $r$, e.g., $r^{(s)}(\bm x)=\|\bm x\|$ for some norm $\|\cdot\|$ on $\RR^{|\mathcal{S}|}$. This corresponds to the well-studied case mentioned above. In this case, if the vector $\bm X_t$ is extreme, all sites are included in the point processes $N_n(u)$ at time $t$.
    \item non-identical site-dependent risk functionals that can be jointly large, e.g., $r^{(s)}(\bm x) = x(s)$. In this case, if the vector $\bm X_t$ is extreme, at least one site is included in the point processes $N_n(u)$ at time $t$.
    \item site-dependent risk functionals which cannot be jointly large, e.g.,
    $$r^{(s)}(\bm x)=x(s) \mathds{1}\{x(s) \geq  x(s'), s' \in S\}.$$
    In this case, if the vector $\bm X_t$ is extreme, one and only one site is included in the point processes $N_n(u)$ at time $t$.
\end{itemize}
So, depending on the choice of the site-dependent risk functionals, one extreme event can lead to one or multiple points in the point process.

We aim to analyze the asymptotic behavior of $N_n(u)$ as $n \to \infty$, which will be studied by using the corresponding Laplace functionals $\Psi_{N_n(u)}$. Here, the Laplace functional of a point process $\Phi$ on 
$[0,1] \times \mathcal{S} \times E_0$ with $E_0 = [0,\infty)^\mathcal{S} \setminus \{\bf 0\}$ is defined as
\begin{align*}
\Psi_{\Phi}(g) :={}& \EE\bigg[\exp\bigg(-\int_{[0,1] \times \mathcal{S} \times E_0}  g(x)\Phi(\mathrm{d}x)\bigg)\bigg], \quad g \in C_K^{+}([0,1] \times \mathcal{S} \times \overline{E}_0),
\end{align*}
where $C_{K}^{+}([0,1] \times \mathcal{S} \times \overline{E}_0)$
denotes the space of non-negative continuous functions $g: [0,1] \times \mathcal{S} \times \overline{E}_0 \to [0,\infty)$ with compact support.

Consequently, the Laplace functional of the point processes $N_n(u)$ is given by
$$ \Psi_{N_n(u)}(g) = \EE\bigg[\exp\bigg( -\sum_{s \in \mathcal{S}}\sum_{j=1}^{n} g(j/n,s,a_{n}^{-1}\bm X_{j})\one\{r^{(s)}(\bm X_{j}) > a_{n}u\} \bigg) \bigg], \quad n \geq 1. $$
According to Proposition 11.1.VIII in \cite{daley2008introduction}, $N_n(u)$ converges weakly to some point process $N(u)$ on $[0,1] \times \mathcal{S} \times \overline{E}_0$ if and only if $\Psi_{N_n(u)}(g) \to \Psi_{N(u)}(g)$ for all $g \in C_{K}^{+}([0,1] \times \mathcal{S} \times \overline{E}_0)$ as $n \to \infty$.

The following condition ensures that $N_n(u)$ can be approximated by $k_n=\lfloor n/ r_{n}\rfloor$ independent clusters consisting of the exceedances within a time series segment of length $r_n$. We assume that $(r_n)_{n \in \NN}$ is an intermediate sequence such that both the length $r_n$ and the number $k_n$ of segments tend to infinity, i.e., $r_n \to \infty$ and $\frac{r_n}{n} \to 0$ as $n \to \infty$.
\medskip

{
\renewcommand{\thecon}{(M)}
\begin{con}
\label{con_1}
For any non-negative, bounded and compactly supported function $f: [0,1] \times \mathcal{S} \times \overline{E}_0 \to [0,\infty)$, we have
$$c_n(f)-d_n(f) \stackrel{n \to \infty}{\longrightarrow} 0,$$ where
$$ c_{n}(f)= \EE \bigg[ \exp\bigg \{ -\sum_{s \in \mathcal{S}} \sum_{j=1}^{n} f(j/n,s,a_{n}^{-1}\bm X_{j}) \bigg \} \bigg] , $$
and 
$$ d_{n}(f)=\prod_{i=1}^{k_n} \EE \bigg[ \exp\bigg\{ -\sum_{s \in \mathcal{S}} \sum_{j=(i-1) r_n + 1}^{i r_n} f(j/n,s,a_{n}^{-1}\bm X_{j}) \bigg\} \bigg].$$
\end{con}
}

{
\renewcommand{\thecon}{(AC)}
\begin{con}
\label{con_2}
For every $u \in (0,\infty)$,
$$\lim_{m \to \infty} \limsup_{n \to \infty} \PP\left(\max_{m \leq |t| \leq r_{n}} \|\bm X_{t}\| > a_{n} u \mid  \|\bm X_{0}\| > a_{n} u\right)=0.$$
\end{con}
}

\begin{rem}
Condition~\ref{con_1} is a mixing condition that generalizes Condition~4.4 in \citet{basrak2009regularly} and can be motivated in the following way: Splitting the time series into $k_n=\lfloor n/ r_{n}\rfloor$ segments of length $r_n$, the Laplace functional $\Psi_{N_n(u)}(g)$ with $g \in C_{K}^{+}([0,1] \times \mathcal{S} \times \overline{E}_0)$ is of the form
 \begin{align*}
 & \EE \bigg[ \exp\bigg\{ -\sum_{s \in \mathcal{S}} \sum_{j=1}^{n} f(j/n,s,a_{n}^{-1}\bm X_{j}) \bigg\} \bigg] \\
&\sim{} \EE \bigg[ \exp\bigg\{ -\sum_{s \in \mathcal{S}} \sum_{i=1}^{k_n} \sum_{j=(i-1) r_n + 1}^{i r_n} f(j/n,s,a_{n}^{-1}\bm X_{j}) \bigg\} \bigg]
\end{align*}
for the function $f(t,s,\bm x) = g(t,s,\bm x) \one\{r^{(s)}(x) > u\}$, which is non-negative, bounded, and has a compact support. If the different segments were independent, the right-hand side would be equal to
\begin{equation*}
\prod_{i=1}^{k_n} \EE \bigg[ \exp\bigg\{ -\sum_{s \in \mathcal{S}} \sum_{j=(i-1) r_n + 1}^{i r_n} f(j/n,s,a_{n}^{-1}\bm X_{j}) \bigg\} \bigg].
\end{equation*}
Thus, Condition \ref{con_1} can be understood as reflecting the eventual independence of the segments of the time series.

The other condition, Condition \ref{con_2} is an anti-clustering condition which coincides with Condition 4.1 in \citet{basrak2009regularly}. Loosely speaking, it ensures that the number of exceedances within the same time segment does not become too large.
\end{rem}

In the following, for a spatially indexed time series $Z=\{\bm Z_t, \, t \in \ZZ\}$, we use the notation 
$$ M_{k,l}^Z = \begin{cases}
    \max_{k \leq j \leq l} \|\bm  Z_j\|,& \ k \leq l, \\
    - \infty, & \ k > l,
\end{cases} \qquad \text{and} \qquad M_{k}^Z := M_{1,k}^Z,
$$
for short.
If $\{\bm X_t, \, t \in \ZZ\}$ is stationary and jointly regular varying with tail process $\{\bm Y_t, \, t \in \ZZ\}$, then, by Theorem 2.1 in \citet{basrak2009regularly}, the quantity given by 
\begin{equation}
\label{extremal_index}
\theta := \lim_{n \to \infty} \lim_{x \to \infty} \PP(M_{r_n}^{X} \leq x \mid \| \bm X_0\| > x) = \PP(M_{1,\infty}^{Y} \leq 1),
\end{equation}
the so-called extremal index, exists. By Proposition 4.2 in \citet{basrak2009regularly}, we also have $\theta > 0$ if Condition~\ref{con_2} holds. Its reciprocal $\theta^{-1}$ can be interpreted as the mean cluster size of extreme values, see \citet{hsing1988exceedance}. For more details on the extremal index, its existence, and interpretation, see \citet{EKM-1997} and \citet{kulik2020heavy}, for example.

Moreover, we define the limit cluster process similar to 
\citet{basrak2009regularly} as the point process
\begin{align}
\label{C}  
C = \sum_{s \in \mathcal{S}}\sum_{j=-\infty}^{\infty}\delta_{(V,s,\bm Z_{j})},
\end{align}
where $V$ is uniformly distributed on $(0,1)$
and, independently of $V$, the time series $\{\bm Z_{j}, \, j \in \mathbb{Z}\}$ has the same distribution as the tail process
$\{\bm Y_{j}, \, j \in \mathbb{Z}\}$ conditional on $\sup_{j \leq -1} \| \bm Y_j\| \leq 1$. In particular, as $\|\bm Z_j\| \leq 1$ and, consequently, $r^{(s)}(\bm Z_j) \leq u \|\bm Z_j\| \leq u $ for $j \leq -1$, we have
$$ \sum_{s \in \mathcal{S}}\sum_{j=-\infty}^{\infty}\delta_{(V,s,\bm Z_{j})} \one\{r^{(s)}(\bm Z_{j}) > u\} = \sum_{s \in \mathcal{S}}\sum_{j=0}^{\infty}\delta_{(V,s,\bm Z_{j})} \one\{r^{(s)}(\bm Z_{j}) > u\}.$$
This observation will allow us to write the weak limit of the point processes $N_n(u)$ from \eqref{N_n} as a superposition of a random number of such clusters, as we will prove in the next theorem.

\begin{thm}
\label{point_process_N_explicit_representation}
Let $\{\bm X_{t}, \, t \in \mathbb{Z}\}$ be a multivariate strictly stationary regularly varying time series with tail process $\{\bm Y_{t}, \, t \in \mathbb{Z}\}$ such that Condition \ref{con_1} and Condition \ref{con_2} hold. In addition, let $N_n(u)$ be the sequence of point processes given in \eqref{N_n} with $u>0$ satisfying the bound in Equation \eqref{eq:risk-bound}.
Then, $N_n(u)$ converges weakly to a point process $N(u)$ on 
$[0,1] \times \mathcal{X}_u$ which has the same distribution as the superposition of a random number of independent and identical copies of cluster processes, i.e.,
\begin{align} \label{eq:cluster-sum}
N(u)=_d{}\sum_{i=1}^{T} \sum_{s \in \mathcal{S}} \sum_{j=0}^{\infty} \delta_{(V_i,s,\bm Z_{ij})} ~\mathds{1}\{r^{(s)}(Z_{ij}) >u\},
\end{align}
where $T$ is a Poisson random variable with mean $\theta$
and $$\sum_{s \in \mathcal{S}} \bigg(\sum_{j=-\infty}^{\infty} \delta_{(V_i,s, \bm Z_{ij})}\bigg)_{i=1}^{\infty}$$ are i.i.d.~copies of the process $C$ defined in \eqref{C} and independent of $T$.
\end{thm}

The result above also allows us to include certain features of extreme events as additional marks to the point process. Mathematically, these features are given by (potentially site-dependent) functionals $\ell^{(s)}: E \to [0, \infty)$, $s \in \mathcal{S}$ such that $\ell^{(s)}$ is continuous at $a_n^{-1}\bm X_t$ for all $n \in \NN$ with probability one.  In principle, the mark $\ell^{(s)}$ can be a chosen very flexibly. For instance, $\ell^{(s)}$ could be identical to the risk functional $r^{(s)}$ that is used to define exceeedances, but it could also be another risk functional or a spatial risk measure such as the percentage of the sites affected by the extreme event, i.e.,
$$ \ell^{(s)}(f) = \frac{1}{|\mathcal{S}|} \sum_{\tilde s \in \mathcal{S}} \one\{r^{(\tilde s)}(f) > u\}, $$
see, e.g., \cite{oesting-huser-2026}.
Including the marks, we obtain the marked point processes
\begin{equation}
\label{M_n}  
M_n(u) = \sum_{s \in \mathcal{S}}\sum_{t=1}^{n}\delta_{(t/n,s,a_{n}^{-1}\bm X_{t},\ell^{(s)}(a_{n}^{-1}\bm X_{t}))} \mathds{1}\{r^{(s)}(\bm X_{t}) > a_{n}u\}, \quad n \geq 1
\end{equation}
on $[0,1]\times \mathcal{X}_{u} \times [0,\infty)$
and investigate in the subsequent theorem its asymptotic behavior as $n \to \infty.$

\begin{thm}
\label{thm_laplacefunctional_marked}
Let $\{\bm X_{t}, \, t \in \mathbb{Z}\}$ be a multivariate strictly stationary regularly varying time series with tail process $\{\bm Y_{t}, \, t \in \mathbb{Z}\}$ such that Condition \ref{con_1} and Condition \ref{con_2} hold. 
Moreover, let $M_n(u)$ be the sequence of point processes given in \eqref{M_n} with $u>0$ satisfying the bound in Equation \eqref{eq:risk-bound}. Then, $M_n(u)$ converges weakly to a point process $M(u)$ on $[0,1] \times \mathcal{X}_u \times [0, \infty)$ that has the same distribution as the superposition of a random number of independent and identical copies of cluster processes, i.e.,
\begin{align*}
&M(u) =_d{} \sum_{i=1}^{T} \sum_{s \in \mathcal{S}} \sum_{j=0}^{\infty} \delta_{(V_i,s,\bm Z_{ij}, \ell^{(s)}(\bm Z_{ij}))} ~\mathds{1}\{r^{(s)}(Z_{ij}) > u\},
\end{align*}
where $T$ is a Poisson random variable with mean $\theta$ and $$\sum_{s \in \mathcal{S}} \bigg(\sum_{j=-\infty}^{\infty} \delta_{(V_i,s, \bm Z_{ij})}\bigg)_{i=1}^{\infty}$$ are i.i.d.~copies of the process $C$ defined in \eqref{C} and independent of $T$.    
\end{thm}

\section{Proofs of the theoretical results}
\label{Proofs}

For any function $g: [0,1] \times \mathcal{S} \times \overline{E}_0 \to [0,\infty)$, we define the thresholded function 
$\thresh{g}_u: [0,1] \times\mathcal{S} \times \overline{E}_0 \to [0,\infty)$ by 
$$ \thresh{g}_u(t,s,\bm x) = g(t,s,\bm x)  \one\{ r^{(s)}(\bm x) > u\}. $$

Using this notation, the Laplace functional $\Psi_{N_n(u)}$ of the point processes $N_n(u)$ can be written as
$$ \Psi_{N_n(u)}(g) = \EE\bigg[\exp\bigg( -\sum_{s \in \mathcal{S}}\sum_{j=1}^{n} \thresh{g}_u(j/n,s,a_{n}^{-1}\bm X_{j}) \bigg) \bigg], \quad n \geq 1. $$

The following lemma and its proof are similar to Theorem 4.3 in \citet{basrak2009regularly}. However, apart from the additional temporal and spatial components, our result also differs in that the continuous function $g$ is multiplied by an indicator function, which makes the resulting thresholded function $\thresh{g}$ discontinuous.

\begin{lem}
\label{Theorem_sim_cond_4.3}
Let $X=\{\bm X_j, \,  j \in \ZZ\}$ be a non-negative, strictly stationary, and jointly regularly varying spatially indexed time series with $\bm X_j = (X_j(s))_{s \in \mathcal{S}}$ and tail process $Y=\{\bm Y_j, \, j \in \ZZ\}$ and assume that Condition \ref{con_2} holds. Then, for all $t \in [0,1]$ and $g \in C_K^+([0,1] \times \mathcal{S} \times \overline{E}_0)$, we obtain
\begin{align*}
& \lim_{n \to \infty} \EE\bigg[\exp\bigg(-\sum_{s \in \mathcal{S}}\sum_{j=1}^{r_n} \thresh{g}_u(t,s,a_{n}^{-1}v \bm X_{j}) \bigg) \, \bigg| \, M_{r_n}^{X} > a_{n} \bigg] \\
={}& 1-\theta^{-1} \EE\bigg[\exp\bigg\{-\sum_{s \in \mathcal{S}} \sum_{j=1}^{\infty} \thresh{g}_u(t,s,v \bm Y_{j}) \bigg\}- \exp\bigg\{-\sum_{s \in \mathcal{S}} \sum_{j=0}^{\infty} \thresh{g}_u(t,s,v \bm Y_{j}) \bigg\}\bigg],
\end{align*}
where $v \in (0,1]$ and the threshold $u$ in $\thresh{g}_u$ is sufficiently large such that Equation~\eqref{eq:risk-bound} holds.
\end{lem}

\begin{proof}
Let $g \in C_K^+([0,1] \times \mathcal{S} \times \overline{E}_0)$, $v \in (0,1]$ and $t \in [0,1]$ be fixed and define
$$\tilde{c}_{n}(k,l)=\exp\bigg (-\sum_{s \in \mathcal{S}} \sum_{j=k}^{l} 
\thresh{g}_u(t,s,a_{n}^{-1}v \bm X_{j}) \bigg),$$
where $k \leq l$ and $k,l \in \mathbb{N}$.
Thus, the statement of the theorem translates into
\begin{align} \label{eq:final-statement}
& \lim_{n \to \infty} \EE\left[\tilde{c}_{n}(1,r_n) \, \Big| \, M_{r_n}^{X} > a_{n}\right]  \\
={}& 1-\theta^{-1} \EE\bigg[\exp\bigg \{-\sum_{s \in \mathcal{S}} \sum_{j=1}^{\infty} \thresh{g}_u(t,s,v \bm Y_{j}) \bigg\}- \exp\bigg\{-\sum_{s \in \mathcal{S}} \sum_{j=0}^{\infty} \thresh{g}_u(t,s,v \bm Y_{j}) \bigg\}\bigg], \nonumber
\end{align}
which we eventually need to show. As an intermediate result, we will first show that
\begin{align} \label{tilde_cn_cond_max_rn}
& \lim_{n \to \infty} \EE\left[\tilde{c}_{n}(1,r_n) \, \Big| \, M_{r_n}^{X} > a_{n} \right]  \nonumber \\
&= \lim_{m \to \infty} \lim_{n \to \infty} \theta_{n}^{-1} \EE\left[\tilde{c}_{n}(-m,m) \cdot 
\one\{M_{-m,-1}^{X} \leq a_{n} \} \, \Big| \, \| \bm X_0\| > a_{n}\right],
\end{align}
with
\begin{equation}
\label{widetilde_theta_n}
\theta_{n} = \frac{\PP(M_{r_n}^{X}>a_{n})}{r_n\PP(\| \bm X_0\| > a_{n})}.    
\end{equation}

To this end, the event $\{M_{r_n}^{X}>a_{n}\}$
is decomposed according to the smallest $j \in \{1,\hdots,r_n\}$ such that $ \| \bm X_j\| > a_{n}$ and hence
\begin{equation*}
\EE\left[\tilde{c}_{n}(1,r_n) \cdot \one\{ M_{r_n}^{X} >a_{n}\}\right]=\sum_{j=1}^{r_n}\EE\left[\tilde{c}_{n}(1,r_n) \cdot \one\{ M_{j-1}^{X} \leq a_{n} < \|\bm X_{j}\|\}\right],
\end{equation*}
where $M^{X}_0 = -\infty$ by definition.
We fix $m>0$ and let $n$ be sufficiently large such that $r_n \geq 2m +1$. 
For $j \in \{m+1,\hdots,r_n-m\}$, we have $\tilde {c}_{n}(1,r_n)=\tilde {c}_{n}(j-m,j+m)$ except if 
$$ \max_{i=1,\ldots, j-m-1,j+m+1,\ldots,r_n} \max_{s \in S} r^{(s)}(\bm X_i) > a_nu/v, $$
which, by \eqref{eq:risk-bound}, implies that 
$$M_{1,j-m-1}^{X}\lor M_{j+m+1,r_n}^{X} > a_{n}/v,$$
and, since  $1/v \geq 1$, even more,
$M_{1,j-m-1}^{X}\lor M_{j+m+1,r_n}^{X} > a_n$,
where $\lor$ denotes the maximum. 
In particular, 
$$ \tilde{c}_{n}(1,r_n) \cdot \one\{M_{j-1}^{X} \leq a_{n} \} \neq \tilde{c}_{n}(j-m,j+m) \cdot \one\{M_{j-m,j-1}^{X} \leq a_{n} \} $$
only if $M_{1,j-m-1}^{X}\lor M_{j+m+1,r_n}^{X} > a_{n}$. Then, for such $j$,
\begin{align}
\label{triangle_j}
& \triangle_{n,m}(j) 
:={} \Big|\EE\left[\tilde{c}_{n}(1,r_n) \cdot \one\{M_{j-1}^{X} \leq a_{n} < \| \bm X_j\| \} \right] \nonumber\\
& \qquad \qquad \qquad - \EE\left[\tilde{c}_{n}(j-m,j+m) \cdot \one\{M_{j-m,j-1}^{X} \leq a_{n}< \| \bm X_j\| \} \right] \Big| \nonumber\\
={}& \Big|\EE\Big[ \Big(\tilde{c}_{n}(1,r_n) \cdot \one\{M_{j-1}^{X} \leq a_{n}\}  \nonumber\\
& \qquad- \tilde{c}_{n}(j-m,j+m) \cdot \one\{M_{j-m,j-1}^{X} \leq a_{n} \} \Big) \one\{\| \bm X_j\| > a_{n}\} \Big] \Big| \nonumber\\
\leq{}& \PP\Big( \tilde{c}_{n}(1,r_n) \cdot \one\{M_{j-1}^{X} \leq a_{n}\} \neq \tilde{c}_{n}(j-m,j+m) \cdot \one\{M_{j-m,j-1}^{X} \leq a_{n}\},  \| \bm X_j\| > a_{n}\Big) \nonumber \\
\leq{}& \PP(M_{1,j-m-1}^{X}\lor M_{j+m+1,r_n}^{X} > a_{n}, \| \bm X_j\| > a_{n}) \nonumber\\
\leq{}& \PP(M_{-r_n,-m-1}^{X}\lor M_{m+1,r_n}^{X} > a_{n}, \| \bm X_0\| > a_{n} ),
\end{align}
where we used the stationarity of the time series for the last inequality.
For $j \notin \{m+1,\hdots,r_n-m\}$, it holds by using the stationarity again,
\begin{align}
\label{triangle_j_2}
\triangle_{n,m}(j) 
={}& \Big|\EE\left[\tilde{c}_{n}(1,r_n) \cdot \one\{M_{j-1}^{X} \leq a_{n}< \| \bm X_j\| \} \right] \nonumber\\
&- \EE\left[\tilde{c}_{n}(j-m,j+m) \cdot \one\{M_{j-m,j-1}^{X} \leq a_{n}< \| \bm X_j\| \} \right] \Big| \nonumber\\
\leq{} & \PP(\| \bm X_j\| > a_{n}) = \PP(\| \bm X_0\| > a_{n}).
\end{align}
By \eqref{triangle_j} and \eqref{triangle_j_2}, it follows that
\begin{align*}
\triangle_{n,m} :={}& \Big|\EE\left[\tilde{c}_{n}(1,r_n) \mid M_{r_n}^{X} > a_{n} \right]\\
&- \tilde{\theta}_{n}^{-1} \EE\left[\tilde{c}_{n}(-m,m) \one\{M_{-m,-1}^{X} \leq a_{n} \} \mid \| \bm X_0\| > a_{n}\right]\Big| \\   
={}&\bigg| \frac{\EE\left[\tilde{c}_{n}(1,r_n) \cdot \one\{ M_{r_n}^{X} > a_{n}\} \right]}{\PP(M_{r_n}^{X} > a_{n} )} \\
&- \frac{r_n \EE\left[\tilde{c}_{n}(-m,m) \one\{M_{-m,-1}^{X} \leq a_{n} < \| \bm X_0\|\}\right]}{\PP(M_{r_n}^{X} > a_{n})} \bigg| \displaybreak[0] \\
={}&\frac{\sum_{j=1}^{m}\triangle_{n,m}(j) +\sum_{j=m+1}^{r_n-m}\triangle_{n,m}(j) +\sum_{j=r_n-m-1}^{r_n}\triangle_{n,m}(j) }{\PP(M_{r_n}^{X} > a_{n} )} \\
\leq{}& \frac{(r_n-2m) \PP(M_{-r_n,-m-1}^{X}\lor M_{m+1,r_n}^{X} > a_{n} , \| \bm X_0\| > a_{n})}{\PP(M_{r_n}^{X} > a_{n})} \\
& + \frac{2m \PP(\| \bm X_0\| > a_{n})}{\PP(M_{r_n}^{X} > a_{n})} \\
\leq{}& \frac{r_n \PP(M_{-r_n,-m-1}^{X}\lor M_{m+1,r_n}^{X} > a_{n}, \| \bm X_0\| > a_{n})
+ 2m \PP(\| \bm X_0\| > a_{n})}{\PP(M_{r_n}^{X} > a_{n} )} \displaybreak[0] \\
={}& \frac{\PP(M_{-r_n,-m-1}^{X}\lor M_{m+1,r_n}^{X} > a_{n}\mid \| \bm X_0\| > a_{n} )+2m/r_n}{\theta_{n}},
\end{align*}
where $\theta_n$  is given in \eqref{widetilde_theta_n}.
By Proposition 4.2 in \citet{basrak2009regularly}, it holds $\lim_{n \to \infty} \theta_{n} = \theta < \infty$\, where $\theta$ is the extremal index defined in Equation~\eqref{extremal_index}.
Furthermore, by Condition \ref{con_2}, we know
\begin{align*}
  &\lim_{m \to \infty} \lim_{n \to \infty} \PP(M_{-r_n,-m}^{X}\lor M_{m,r_n}^{X} > a_{n} \mid \| \bm X_0\| > a_{n}) \nonumber\\
  &=\lim_{m \to \infty} \limsup_{n \to \infty} \PP\left(\max_{m \leq |j| \leq r_{n}} \|\bm X_{j}\| > a_{n} \, \Big| \,  \|\bm X_{0}\| > a_{n} \right)=0,
\end{align*}
and we can conclude that $$\lim_{m \to \infty} \limsup\limits_{n\rightarrow \infty} \triangle_{n,m}=0.$$
Therefore,
\begin{align*}
 &\lim_{n \to \infty} \EE\left[\tilde{c}_{n}(1,r_n) \, \Big| \, M_{r_n}^{X} > a_{n} \right]  \nonumber \\
   &=  \lim_{m \to \infty} \lim_{n \to \infty} \tilde{\theta}_{n}^{-1} \EE\left[\tilde{c}_{n}(-m,m)\cdot \one\{M_{-m,-1}^{X} \leq a_{n} \} \, \Big| \, \| \bm X_0\| > a_{n} \right],
\end{align*}
i.e., the intermediate result \eqref{tilde_cn_cond_max_rn} is verified.
To calculate the inner limit, we note that, from the representation $\bm Y_t = P_\alpha \bm \Theta_t$, $t \in \ZZ$, and the fact that $$\PP( r^{(s)}(\bm \Theta_{t}) = c/z)=0, \quad c>0,$$ for almost every $z > 1$ as the distribution of $r^{(s)}(\bm \Theta_{t})$ may have at most countably many atoms, it follows that 
\begin{align*}
\PP(r^{(s)}(\bm Y_{t}) = u/v)&=\PP(P_\alpha r^{(s)}(\bm \Theta_{t}) = u/v) \\
& ={} \int_{1}^{\infty} \PP( r^{(s)}(\bm \Theta_{t}) = u/(vz)) \alpha z^{-\alpha-1} \,\mathrm{d}z= 0, \quad t \in \ZZ, 
\end{align*}
and, with analogous arguments,
$$ \PP(M_{-m,-1}^{Y} =1) =0, \quad m \in \NN.$$

Consequently, regular variation implies that
\begin{align}
\label{rewrite_with_tailprocess}
&\lim_{n \to \infty} \tilde{\theta}_{n}^{-1} \EE\left[\tilde{c}_{n}(-m,m) \one\{M_{-m,-1}^{X} \leq a_{n}\} \mid \| \bm X_0\| > a_{n}\right] \nonumber\\
&=\theta^{-1}\EE\bigg[\exp\bigg \{-\sum_{s \in \mathcal{S}} \sum_{j=-m}^{m} \thresh{g}_u(t,s,v \bm Y_{j}) \bigg\} \cdot \one\{M_{-m,-1}^{Y} \leq 1\}\bigg]. 
\end{align}
Then, by \eqref{tilde_cn_cond_max_rn} and \eqref{rewrite_with_tailprocess},
\begin{align}
\label{result_before_bounded_convergence_1}
 &\lim_{n \to \infty} \EE\left[\tilde{c}_{n}(1,r_n) \mid M_{r_n}^{X} > a_{n} \right]  \nonumber \\
   &= \lim_{m \to \infty} \theta^{-1} \EE\bigg[\exp\bigg \{-\sum_{s \in \mathcal{S}} \sum_{j=-m}^{m} \thresh{g}_u(t,s,v \bm Y_{j}) \bigg\} \cdot \one\{M_{-m,-1}^{Y} \leq 1\}\bigg]
\end{align}
and, rewriting the right-hand side of \eqref{result_before_bounded_convergence_1} further,
\begin{align} \label{result_before_bounded_convergence_2}
 &\lim_{n \to \infty} \EE\left[\tilde{c}_{n}(1,r_n) \mid M_{r_n}^{X} > a_{n}\right]  \nonumber \\
   & ={} \lim_{m \to \infty} \theta^{-1} a_{m,m} 
   ={} \lim_{m \to \infty} \theta^{-1} 
   \left[ a_{m,0}+\sum\nolimits_{q=1}^{m}(a_{m,q}-a_{m,q-1}) \right],
\end{align}
where 
$$a_{m,q}=\EE\bigg[\exp\bigg \{-\sum_{s \in \mathcal{S}} \sum_{j=-q}^{m} \thresh{g}_u(t,s,v \bm Y_{j}) \bigg\} \cdot \one\{M_{-q,-1}^{Y} \leq 1\}\bigg]$$
for  $m \in \mathbb{N}$ and $q \in \{0,\hdots,m\}$.
For fixed $m$ and $q$, we can write $$a_{m,q}-a_{m,q-1}=\EE[g_{m,q}( \bm Y_{-q},\hdots, \bm Y_{m})],$$
where 
\begin{align}
\label{g_m,q}
& g_{m,q}( \bm y_{-q},\hdots,\bm y_{m}) \nonumber\\
={}&  \exp\bigg \{-\sum_{s \in \mathcal{S}} \sum_{j=-q}^{m} \thresh{g}_u(t,s,v \bm y_{j}) \bigg\} \mathds{1}\left\{\max_{-q \leq j \leq -1}\| \bm y_j\| \leq 1 \right\} \nonumber \\
&- \exp\bigg \{-\sum_{s \in \mathcal{S}} \sum_{j=-q+1}^{m} \thresh{g}_u(t,s,v \bm y_{j}) \bigg\} \mathds{1}\left\{\max_{-q+1 \leq j \leq -1}\|\bm y_j\| \leq 1 \right\}.
\end{align}

We now approximate the indicator functions $\one_{(-\infty,1]}$ and $\one_{(1,\infty]}$ by continuous functions $f_n^{-}, f_n^{+}: \RR \to [0,1]$ such that
$f_n^{-} \to \one_{(-\infty,1]}$ and $f_n^{+} \to \one_{(1,\infty)}$ pointwise as $n \to \infty$. 

Then, we apply the bounded convergence theorem to the functions that map $(\bm y_{-q},\ldots,\bm y_m)$ to 
\begin{align*}
& \exp\bigg \{-\sum_{s \in \mathcal{S}} \sum_{j=-q}^{m} g(t,s,v \bm y_{j})f_n^{+}(r^{(s)}( v \bm y_j)) \bigg\} f_n^{-}\bigg(\max_{-q \leq j \leq -1}\| \bm y_j\|\bigg) \nonumber \\
& - \exp\bigg \{-\sum_{s \in \mathcal{S}} \sum_{j=-q+1}^{m} g(t,s,v \bm y_{j})f_n^{+}(r^{(s)}( v \bm y_j)) \bigg\} f_n^{-}\bigg(\max_{-q+1 \leq j \leq -1}\| \bm y_j\|\bigg),
\end{align*}
since these functions converge pointwise to \eqref{g_m,q}, are continuous and uniformly bounded by 1. 
Therefore, the convergence in Theorem 3.1(ii) in \citet{basrak2009regularly} is also true for $g_{m,q}$, that is,
\begin{equation*}
\EE\left[ g_{m,q}\left(\bm Y_{-q},\hdots,\bm Y_{m} \right) \right]=\int_{0}^{\infty}\EE\left[ g_{m,q}\left(z \bm \Theta_{0},\hdots z \bm \Theta_{m+q}\right)\one\{z \| \bm \Theta_{q} \| > 1\}\right]\alpha z^{-\alpha-1} \,\mathrm{d}z,
\end{equation*}
where $\{\bm \Theta_t, \, t \in \ZZ\}$ is the corresponding spectral tail process.
Note that $g_{m,q}(\bm y_{-q},\hdots,\bm y_{m})=0$ if $\|v \bm y_{-q}\| \leq  1$ since, in this case,
$$ r^{(s)}( v \bm y_{-q}) \leq u \|v\bm y_{-q}\|  \leq u$$
for all $s \in \mathcal{S}$ by \eqref{eq:risk-bound}. Therefore,
$$\EE\left[ g_{m,q}\left(z \bm \Theta_{0},\hdots z \bm \Theta_{m+q}\right)\one\{z \| \bm \Theta_{q} \| > 1\} \right]=0$$
if $v z \| \bm \Theta_{0} \| \leq 1$.
Since $\PP(\|\bm \Theta_{0}\|=1)=1$, it follows
$$\int_{0}^{b}\EE\left[ g_{m,q}\left(z \bm \Theta_{0},\hdots z \bm \Theta_{m+q}\right)\one\{z \| \bm \Theta_{q} \| \leq 1\}\right]\alpha z^{-\alpha-1} \,\mathrm{d}z = 0 $$
for all $b \in [0,1/v]$. In particular,
$$\int_{0}^{1}\EE\left[ g_{m,q}\left(z \bm \Theta_{0},\hdots z \bm \Theta_{m+q}\right)\one\{z \| \bm \Theta_{q} \| \leq 1\}\right]\alpha z^{-\alpha-1} \,\mathrm{d}z = 0 $$
and, consequently,
\begin{align*}
 &a_{m,q}-a_{m,q-1}=\EE\left[ g_{m,q}\left(\bm Y_{-q},\hdots,\bm Y_{m} \right) \right] \nonumber \\
 ={}&\int_{0}^{\infty}\EE\left[ g_{m,q}\left(z \bm \Theta_{0},\hdots z \bm \Theta_{m+q}\right)\one\{z \| \bm \Theta_{q} \| > 1\}\right]\alpha z^{-\alpha-1} \,\mathrm{d}z\nonumber \\
 ={}&\int_{1}^{\infty}\EE\left[ g_{m,q}\left(z \bm \Theta_{0},\hdots z \bm \Theta_{m+q}\right)\one\{z \| \bm \Theta_{q} \| > 1\}\right]\alpha z^{-\alpha-1} \,\mathrm{d}z \\
={}& \int_{1}^{\infty} \EE\bigg[\exp\bigg \{-\sum_{s \in \mathcal{S}} \sum_{j=0}^{m+q} \thresh{g}_u(t,s,vz \bm \Theta_{j}) \bigg\} \cdot \one\{M_{0,q-1}^{z\Theta} \leq 1 < z \| \bm \Theta_{q} \|\}\bigg] \alpha z^{-\alpha-1}\,\mathrm{d}z\\
&-\int_{1}^{\infty} \EE\bigg[\exp\bigg \{-\sum_{s \in \mathcal{S}} \sum_{j=1}^{m+q} \thresh{g}_u(t,s,vz \bm \Theta_{j})\bigg\} \cdot \one\{ M_{1,q-1}^{z\Theta} \leq 1 < z \| \bm \Theta_{q} \| \}\bigg]  \alpha z^{-\alpha-1} \,\mathrm{d}z.
\end{align*}

Arguing similarly to the proof of Theorem 4.3 in \citet{basrak2009regularly}, we can approximate $a_{m,q}-a_{m,q-1}$ by
\begin{align*}
&b_{m,q} \\
={}& \int_{1}^{\infty} \EE\bigg[\exp\bigg \{-\sum_{s \in \mathcal{S}} \sum_{j=0}^{m} \thresh{g}_u(t,s,vz \bm \Theta_{j}) \bigg\} \cdot \one\{M_{0,q-1}^{z\Theta} \leq 1 < z \| \bm \Theta_{q} \|\}\bigg] \alpha z^{-\alpha-1}\,\mathrm{d}z\\
&-\int_{1}^{\infty} \EE\bigg[\exp\bigg \{-\sum_{s \in \mathcal{S}} \sum_{j=1}^{m} \thresh{g}_u(t,s,vz \bm \Theta_{j}) \bigg\} \cdot \one\{ M_{1,q-1}^{z\Theta} \leq 1 < z \| \bm \Theta_{q} \| \}\bigg]  \alpha z^{-\alpha-1} \,\mathrm{d}z.
\end{align*}
More precisely, the equality
$$\sum_{s \in \mathcal{S}} \sum_{j=l}^{m+q} \thresh{g}_u(t,s,vz \bm \Theta_{j})=\sum_{s \in \mathcal{S}} \sum_{j=l}^{m} \thresh{g}_u(t,s,vz \bm \Theta_{j})$$
holds for $l \in \{0,1\}$ except if $z \|\Theta_{j}\|>1/v$ for some $j \geq m+1$.
Therefore,
\begin{align*}
&\bigg|(a_{m,q}-a_{m,q-1})- b_{m,q}\bigg|\\
\leq{}& \int_{1}^{\infty} \PP(M_{m+1,\infty}^{z\Theta}>1/v, M_{0,q-1}^{z\Theta} \leq 1 < z \| \bm \Theta_{q} \|) \alpha z^{-\alpha-1}\,\mathrm{d}z \\
&+\int_{1}^{\infty}\PP(M_{m+1,\infty}^{z\Theta}>1/v, M_{1,q-1}^{z\Theta} \leq 1 < z \| \bm \Theta_{q} \|) \alpha z^{-\alpha-1}\,\mathrm{d}z \\
\leq{}& 2 \int_{1}^{\infty} \PP(M_{m+1,\infty}^{z\Theta}>1/v, M_{1,q-1}^{z\Theta} \leq 1 < z \| \bm \Theta_{q} \|) \alpha z^{-\alpha-1}\,\mathrm{d}z.
\end{align*}

Thus, 
\begin{align*}
    &\lim_{m \to \infty} \bigg|\sum_{q=1}^{m} (a_{m,q}-a_{m,q-1})-\sum_{q=1}^{m} b_{m,q}\bigg| \\
   \leq{}&  \lim_{m \to \infty} 2 \int_{1}^{\infty} \PP(M_{m+1,\infty}^{z\Theta}>1/v) \alpha z^{-\alpha-1}\,\mathrm{d}z \\
     ={}&  \lim_{m \to \infty} 2 \PP(M_{m+1,\infty}^{Y}>1/v) 
    = 2\PP\left(\bigcap\nolimits_{m=1}^\infty \{M_{m+1,\infty}^Y > 1/v\} \right) \\
    ={}& 2\PP\left( \limsup_{m \to \infty} \bm Y_m \geq 1/v\right) = 0,
\end{align*}
as $\bm Y_m \to 0$ a.s.~by Proposition 4.2 in \citet{basrak2009regularly}. Together with~\eqref{result_before_bounded_convergence_2}, this implies
\begin{equation} \label{result_after_bounded_convergence}
 \lim_{n \to \infty} \EE\left[\tilde{c}_{n}(1,r_n) \mid M_{r_n}^{X} > a_n\right]  
    ={} \lim_{m \to \infty}  \theta^{-1} 
   \left[ a_{m,0}+\sum\nolimits_{q=1}^{m} b_{m,q} \right]
\end{equation}
Since $\PP(\|\bm \Theta_{0}\|=1)=1,$
\begin{align*}
 &a_{m,0}= \EE\bigg[\exp\bigg \{-\sum_{s \in \mathcal{S}} \sum_{j=0}^{m} 
 \thresh{g}_u(t,s,v \bm Y_{j}) \bigg\} \bigg] \nonumber\\ 
 &= \int_{1}^{\infty} \EE\bigg[\exp\bigg \{-\sum_{s \in \mathcal{S}} \sum_{j=0}^{m} \thresh{g}_u(t,s,vz \bm \Theta_{j}) \bigg\} \bigg] \alpha z^{-\alpha-1} \,\mathrm{d}z \nonumber\\ 
 &= \int_{1}^{\infty} \EE\bigg[\exp\bigg \{-\sum_{s \in \mathcal{S}} \sum_{j=0}^{m} \thresh{g}_u(t,s,vz \bm \Theta_{j})\bigg\} \cdot \one\{z \|\bm \Theta_{0}\| > 1\}\bigg] \alpha z^{-\alpha-1} \,\mathrm{d}z
\end{align*}
Therefore, with \eqref{result_after_bounded_convergence},
\begin{align*}
& \lim_{n \to \infty} \EE\left[\tilde{c}_{n}(1,r_n) \mid M_{r_n}^{X} > a_n \right] \\
={}& \lim_{m \to \infty} \theta^{-1} \bigg( \int_{1}^{\infty} \EE\bigg[\exp\bigg \{-\sum_{s \in \mathcal{S}} \sum_{j=0}^{m} \thresh{g}_u(t,s,vz \bm \Theta_{j}) \bigg\} \cdot \one\{z M_{0,m}^{\Theta} > 1\}\bigg] \alpha z^{-\alpha-1} \,\mathrm{d}z\\
&\qquad \qquad -\int_{1}^{\infty} \EE\bigg[\exp\bigg \{-\sum_{s \in \mathcal{S}} \sum_{j=1}^{m} \thresh{g}_u(t,s,vz \bm \Theta_{j}) \bigg\} \cdot \one\{z M_{1,m}^{\Theta} > 1\}\bigg] \alpha z^{-\alpha-1} \,\mathrm{d}z \bigg).
\end{align*}
To calculate the limit, let 
\begin{align*}
f_{m}(z)={}& \EE\bigg[\exp\bigg \{-\sum_{s \in \mathcal{S}} \sum_{j=0}^{m} \thresh{g}_u(t,s,vz \bm \Theta_{j})\bigg\} \cdot \one\{M_{0,m}^{z \Theta} > 1\}\bigg] \alpha z^{-\alpha-1}\\
&- \EE\bigg[\exp\bigg \{-\sum_{s \in \mathcal{S}} \sum_{j=1}^{m} \thresh{g}_u(t,s,vz \bm \Theta_{j}) \bigg\} \cdot \one\{M_{1,m}^{z \Theta} > 1\}\bigg] \alpha z^{-\alpha-1}.
\end{align*}
As
\begin{align*}
&\bigg|\EE\bigg[\exp\bigg\{-\sum_{s \in \mathcal{S}} \sum_{j=0}^{m} \thresh{g}_u(t,s,vz \bm \Theta_{j}) \bigg\} \cdot \one\{M_{0,m}^{z\Theta} > 1\}\bigg] \\
&- \EE\bigg[\exp\bigg \{-\sum_{s \in \mathcal{S}} \sum_{j=1}^{m} \thresh{g}_u(t,s,v z\bm \Theta_{j})\bigg\} \cdot \one\{M_{1,m}^{z\Theta} > 1\}\bigg] \bigg| \leq 1,
\end{align*}
we have $|f_{m}(r)| \leq \alpha z^{-\alpha-1}$ and 
$\int_{1}^{\infty}\alpha z^{-\alpha-1} = 1 < \infty.$
Consequently, we can apply the dominated convergence theorem twice to obtain
\begin{align} \label{result_after_dominated_convergence}
    & \lim_{n \to \infty} \EE\left[\tilde{c}_{n}(1,r_n) \mid M_{r_n}^{X} > a_n \right]  \nonumber \\
    ={}& \theta^{-1} \int_{1}^{\infty} \EE\bigg[\exp\bigg \{-\sum_{s \in \mathcal{S}} \sum_{j=0}^{\infty} \thresh{g}_u(t,s,vz \bm \Theta_{j}) \bigg\} \cdot 
    \one\{ M_{0,\infty}^{z \Theta} > 1\}\bigg] \alpha z^{-\alpha-1} \,\mathrm{d}z \nonumber\\
    &- \theta^{-1} \int_{1}^{\infty} \EE\bigg[\exp\bigg \{-\sum_{s \in \mathcal{S}} \sum_{j=1}^{\infty} \thresh{g}_u(t,s,vz \bm \Theta_{j})\bigg\} \cdot \one\{M_{1,\infty}^{z \Theta} > 1\}\bigg] \alpha z^{-\alpha-1} \,\mathrm{d}z \nonumber \displaybreak[0] \\
    ={}& \theta^{-1} \EE\bigg[\exp\bigg \{-\sum_{s \in \mathcal{S}} \sum_{j=0}^{\infty} \thresh{g}_u(t,s,v \bm Y_{j}) \bigg\} \cdot 
    \one\{ M_{0,\infty}^{Y} > 1\}\bigg] \nonumber\\
    &- \theta^{-1} \EE\bigg[\exp\bigg \{-\sum_{s \in \mathcal{S}} \sum_{j=1}^{\infty} \thresh{g}_u(t,s,v \bm Y_{j})\bigg\} \cdot \one\{M_{1,\infty}^{Y} > 1\}\bigg] \nonumber \\
    ={}& \theta^{-1} \EE\bigg[\exp\bigg\{-\sum_{s \in \mathcal{S}} \sum_{j=0}^{\infty} \thresh{g}_u(t,s,v \bm Y_{j}) \bigg\}\bigg] \nonumber\\
    &- \theta^{-1} \EE\bigg[\exp\bigg \{-\sum_{s \in \mathcal{S}} \sum_{j=1}^{\infty} \thresh{g}_u(t,s,v \bm Y_{j})\bigg\} \cdot \one\{M_{1,\infty}^{Y} > 1\}\bigg], 
\end{align}
where the last equality holds because
$$ M_{0,\infty}^{Y} \geq \|\bm Y_0\| > 1 $$
with probability one. If $M_{1,\infty}^{Y} \leq 1$, then, for all $j \in \{1,\hdots, \infty\}$,
we have $\|\bm Y_j\| \leq 1$ which implies
$$ r^{(s)}( v \bm Y_j) =  v  r^{(s)}(\bm Y_j) \leq u v \|\bm Y_j\| \leq u v \leq u  $$
for all $s \in \mathcal{S}$ and, consequently,
$\thresh{g}_u(t,s,v \bm Y_{j})=0$. Therefore, with the last identity in \eqref{extremal_index},
\begin{align*}  
&\EE\bigg[\exp\bigg\{-\sum_{s \in \mathcal{S}} \sum_{j=1}^{\infty} \thresh{g}_u(t,s,v \bm Y_{j})\bigg\} \cdot \one\{M_{1,\infty}^{Y} \leq 1\}\bigg] = \PP(M_{1,\infty}^{Y} \leq 1)=\theta.
\end{align*}
Plugging this identity into \eqref{result_after_dominated_convergence}, we obtain
\begin{align*}
    & \lim_{n \to \infty} \EE\left[\tilde{c}_{n}(1,r_n) \mid M_{r_n}^{X} > a_n\right]  \\
     ={}& \theta^{-1} \bigg(\EE\bigg[\exp\bigg \{-\sum_{s \in \mathcal{S}} \sum_{j=0}^{\infty}\thresh{g}_u(t,s,v \bm Y_{j}) \bigg\} \bigg] 
    - \EE\bigg[\exp\bigg \{-\sum_{s \in \mathcal{S}} \sum_{j=1}^{\infty} \thresh{g}_u(t,s,v \bm Y_{j}) \bigg\} \bigg]\bigg) \nonumber\\
     &+ \theta^{-1} \EE\bigg[\exp\bigg \{-\sum_{s \in \mathcal{S}} 
     \sum_{j=1}^{\infty} \thresh{g}_u(t,s,v \bm Y_{j}) \bigg\}  \cdot \one\{M_{1,\infty}^{Y} \leq  1\}\bigg] \displaybreak[0] \nonumber\\
   ={}&1-\theta^{-1} \EE\bigg[\exp\bigg \{-\sum_{s \in \mathcal{S}} \sum_{j=1}^{\infty} \thresh{g}_u(t,s,v \bm Y_{j}) \bigg\}- \exp\bigg\{-\sum_{s \in \mathcal{S}} \sum_{j=0}^{\infty} \thresh{g}_u(t,s,v \bm Y_{j}) \bigg\}\bigg],
\end{align*}
that is \eqref{eq:final-statement}, which closes the proof.
\end{proof}

The following lemma and its proof are based on Theorem 4.2 in 
\citet{hsing1988exceedance} and Lemma \ref{Theorem_sim_cond_4.3}. 
We derive the weak limit of the point processes $N_n(u)$ in terms of Laplace functionals.

\begin{lem}
\label{thm_laplacefunctional_neu}
Let $\{\bm X_j, \,  j \in \ZZ\}$ be a non-negative, strictly stationary and jointly regularly varying time series with tail process $Y=\{\bm Y_j, \, j \in \ZZ\}$ and spectral tail process $\{\bm \Theta_{j}, \, j \in \mathbb{Z}\}$
such that Condition \ref{con_1} and Condition \ref{con_2} hold. 
Furthermore, let $u>0$ be chosen such that \eqref{eq:risk-bound} holds. Then, the point process $N_{n}(u)$ given in \eqref{N_n} converges weakly to a point process $N(u)$ on $[0,1] \times \mathcal{X}_{u} \subset [0,1] \times \mathcal{S} \times \overline{E}_0$ as $n \to \infty$. The Laplace functional of the limit process $N(u)$ is given by
\begin{align}
    \Psi_{N(u)}(g) &=  \exp\bigg(-\int_{0}^{1}  \EE\bigg[\exp\bigg\{-\sum_{s \in \mathcal{S}} \sum_{j=1}^{\infty}\thresh{g}_u(t,s,\bm Y_{j})\bigg\} \nonumber \\
	& \qquad \qquad \qquad \qquad \quad -\exp\bigg\{-\sum_{s \in \mathcal{S}} \sum_{j=0}^{\infty}\thresh{g}_u(t,s,\bm Y_{j}) \bigg\}\bigg] \,\mathrm{d}t \bigg),
    \label{eq:laplace-tailprocess}
\end{align}
with $g \in C_K^+([0,1] \times \mathcal{S} \times \overline{E}_0)$.
It can be rewritten in terms of the spectral tail process as
\begin{align}
  \Psi_{N(u)}(g) & = \exp\bigg(-\int_{0}^{1} \int_{0}^{\infty}\EE\bigg[  
\exp\bigg \{-\sum_{s \in \mathcal{S}} \sum_{j=1}^{\infty} \thresh{g}_u(t,s,v \bm \Theta_{j}) \bigg\} \nonumber \\
&\hspace*{35mm}- \exp\bigg\{-\sum_{s \in \mathcal{S}} \sum_{j=0}^{\infty} \thresh{g}_u(t,s,v \bm \Theta_{j}) \bigg\} \bigg]\alpha v^{-\alpha-1} \,\mathrm{d}v \,\mathrm{d}t \bigg). 
    \label{eq:laplace-spectraltailprocess}
\end{align}
\end{lem}
\begin{proof}
Let $R_n$ be a step function on $(0,1)$ defined by
$$R_n(t) = \left\{
\begin{array}{ll}
1-\EE\exp\left(-\sum_{s}\sum_{j=(i-1)r_n+1}^{ir_n} \thresh{g}_u(j/n,s,a_{n}^{-1} \bm X_{j})\right), & t \in T_1^{(i)}, 1 \leq i \leq k_n, \\
0, & \, t=0 \text{ or } t \in T_2, \\
\end{array}
\right. $$
where $T_1^{(i)}=\left(\frac{(i-1)r_n}{n},\frac{ir_n}{n}\right]$ and $T_2=\left(\frac{k_nr_n}{n},1\right]$.
In addition, let
$$\tilde{R}_n(t) = \left\{
\begin{array}{ll}
1-\EE\exp\left(-\sum_{s}\sum_{j=(i-1)r_n+1}^{ir_n} \thresh{g}_u(t,s,a_{n}^{-1} \bm X_{j})\right), & t \in T_1^{(i)}, 1 \leq i \leq k_n, \\
0, & \, t=0 \text{ or } t \in T_2. \\
\end{array}
\right. $$
Then, by stationarity, it holds $$\tilde{R}_n(t)=1-\EE\exp\left(- \sum_s \sum_{j=1}^{r_n} \thresh{g}_u(t,s,a_{n}^{-1} \bm X_{j})\right)$$  for $0 < t \leq \frac{k_n r_n}{n}$. 
Using also Condition \ref{con_1}, we obtain
\begin{align*} 
& \lim_{n \to \infty} \EE \bigg[ \exp\bigg\{ -\sum_{s \in \mathcal{S}} \sum_{j=1}^{n} \thresh{g}_u(j/n,s,a_{n}^{-1} \bm X_{j}) \bigg\} \bigg] \\
&= \lim_{n \to \infty} \prod_{i=1}^{k_n}\EE \bigg[ \exp\bigg\{ -\sum_{s \in \mathcal{S}} \sum_{j=(i-1)r_n+1}^{ir_n} \thresh{g}_u(j/n,s,a_{n}^{-1} \bm X_{j}) \bigg\} \bigg] \\
&= \lim_{n \to \infty} \exp\bigg(\sum_{i=1}^{k_n}\log\bigg(\EE \bigg[ \exp\bigg\{ -\sum_{s \in \mathcal{S}} \sum_{j=(i-1)r_n+1}^{ir_n} \thresh{g}_u(j/n,s,a_{n}^{-1} \bm X_{j}) \bigg\} \bigg]\bigg) \bigg) \displaybreak[0] \\
&= \lim_{n \to \infty} \exp\bigg(\frac{n}{r_n}\sum_{i=1}^{k_n} \frac{r_n}{n}\log\bigg(1-\bigg(1-\EE \bigg[ \exp\bigg\{ -\sum_{s \in \mathcal{S}} \sum_{j=(i-1)r_n+1}^{ir_n} \thresh{g}_u(j/n,s,a_{n}^{-1} \bm X_{j}) \bigg\} \bigg]\bigg)\bigg) \bigg) \\
&=  \lim_{n \to \infty} \exp\bigg(\frac{n}{r_n}\int_{0}^{1} \log\bigg(1-R_n(t)\bigg) \,\mathrm{d}t \bigg).
\end{align*}

Condition \ref{con_2} implies Proposition 4.2 of \citet{basrak2009regularly}, which gives 
$$\theta=\lim_{n \to \infty} \frac{\PP(M_{r_n}^{X}>a_{n}
)}{r_n\PP( \| \bm X_0\|> a_n)},$$
and thus by \eqref{a_n_sequence}, it is
\begin{align}
\label{theta_limes_finite}
\lim_{n \to \infty} \frac{n}{r_n} \PP(M_{r_n}^{X}> a_{n})
={}&\lim_{n \to \infty} n \PP( \| \bm X_0\|> a_n) \cdot \frac{\PP(M_{r_n}^{X}> a_{n}
)}{r_n\PP( \| \bm X_0\|> a_{n})} = \theta < \infty.
\end{align}
Since $\{\bm X_j, \,  j \in \ZZ\}$ is stationary, it holds
\begin{align}
\label{M_r_n>a_n}
 \PP(M_{r_n}^{X}>a_{n}) ={}& \PP\bigg(\sum_{s \in \mathcal{S}}  \sum_{j=1}^{r_n} \one\{X_j(s) > a_{n}\} >0\bigg) \nonumber \\
 ={}&   \PP\bigg(\sum_{s \in \mathcal{S}}  \sum_{j=(i-1)r_n+1}^{ir_n} \one\{X_j(s) > a_{n}\} >0\bigg).
\end{align}
We note that, by definition, $\thresh{g}_u(j/n,s,a_n^{-1} \bm X_j)>0$ only if 
    $r_s(\bm X_j) > a_nu$, which also implies $\|\bm X_j\| > a_n$. Thus, we can conclude from \eqref{theta_limes_finite} and \eqref{M_r_n>a_n} that $\frac{n}{r_n}R_{n}(t)$ is uniformly bounded in $t$  because 
\begin{align}
\label{n_rn_R_n(t)}
 &\frac{n}{r_n}R_{n}(t) \nonumber \\
 ={}& \frac{n}{r_n} \EE\bigg( \bigg[1-\exp\bigg(- \sum_{s \in \mathcal{S}} \sum_{j=(i-1)r_n+1}^{ir_n} \thresh{g}_u(j/n,s,a_{n}^{-1} \bm X_{j}) \bigg) \bigg] \cdot \nonumber\\
 & \hspace*{40mm} \one\bigg\{\sum_{s \in \mathcal{S}}  \sum_{j=(i-1)r_n+1}^{ir_n} \one\{\|\bm X_j\| > a_n\} >0 \bigg\}\bigg)  \nonumber \\
={}& \frac{n}{r_n} \PP(M_{r_n}^{X}>a_{n}) \cdot \nonumber \\
& \quad \bigg(1-\EE\bigg(\exp\bigg(- \sum_{s \in \mathcal{S}} \sum_{j=(i-1)r_n+1}^{ir_n} \thresh{g}_u(j/n,s,a_{n}^{-1} \bm X_{j}) \bigg) \, \bigg| \, M_{r_n}^{X}> a_{n}\bigg) \bigg) \nonumber \\
 \leq{}& \frac{n}{r_n} \PP(M_{r_n}^{X}>a_{n}) \to \theta < \infty.
\end{align}
Using the fact that $r_n/n \to 0$, this implies that
\begin{align*}
 0 \leq R_{n}(t) = \frac{r_n}{n} \frac{n}{r_n} R_{n}(t) \leq \frac{r_n}{n}  \frac{n}{r_n} \PP(M_{r_n}^{X} > a_{n})\to 0, \quad  n \to \infty,
\end{align*}
i.e., $R_{n}(t) \to 0$ uniformly in $t$. 

Then, as in the proof of Theorem 4.2 in \citet{hsing1988exceedance}, we define $$\psi(x)=-\log(1-x)-x, \quad x \in [0,1),$$ such that
$\psi(x) \sim \frac{x^{2}}{2}$ as $x \to \infty$.
Therefore, for large $n$ it holds
$$|\psi\left(R_{n}(t)\right)| \leq R_{n}^{2}(t) \text{ for all } t \in [0,1].$$

Using again \eqref{n_rn_R_n(t)} and $r_n/n \to 0$, this implies
$$\frac{n}{r_n} \int_{0}^{1} |\psi\left(R_{n}(t)\right)| \,\mathrm{d}t \leq \frac{r_n}{n} \int_{0}^{1} \left( \frac{n}{r_n} R_{n}(t)\right)^{2} \,\mathrm{d}t \to 0,  \quad  n \to \infty,$$
and consequently, 
$$\lim_{n \to \infty} \exp\left(\frac{n}{r_n}\int_{0}^{1} \log\left(1-R_n(t)\right) \,\mathrm{d}t \right) = \lim_{n \to \infty} \exp\left(-\frac{n}{r_n}\int_{0}^{1} R_n(t) \,\mathrm{d}t \right).$$

In the next step, we show that
$$\lim_{n \to \infty} \frac{n}{r_n}|R_{n}(t)- \tilde{R}_n(t)|=0 \quad \text{for all } t \in [0,1].$$
As the support of $g$ is compact, the continuous function $g$ is also uniformly continuous. Consequently, for any $\varepsilon > 0$, there exists some $\delta > 0$ such that
$$ |g(t_1, s, \bm x) - g(t_2,s, \bm x) | < \varepsilon $$
for all $t_1,t_2 \in [0,1]$, $s \in \mathcal{S}$, $\bm  x \in \overline{E}_0$ with 
$|t_1 - t_2| < \delta$. Choosing $n$ sufficiently large such that $r_n / n < \delta$ guarantees that
\begin{align*}
(g(t, s, a_n^{-1} \bm X_j) - \varepsilon)_+ \leq{}& g(t, s, a_n^{-1} \bm X_j) \\
\leq{}& g(t, s, a_n^{-1} \bm X_j) + \varepsilon \leq{} (g(t, s, a_n^{-1} \bm X_j) - \varepsilon)_+ + 2 \varepsilon\\
\text{and }
(g(t, s, a_n^{-1} \bm X_j) - \varepsilon)_+ \leq{}& g(j/n, s, a_n^{-1} \bm X_j) \\\leq{}& g(t, s, a_n^{-1} \bm X_j) + \varepsilon \leq{} (g(t, s, a_n^{-1} \bm X_j) - \varepsilon)_+ + 2 \varepsilon
\end{align*}
with probability one for all $t \in [0,1]$, $s \in \mathcal{S}$, $j \in \{(i-1)r_n+1, \ldots, i r_n\}$ where $i$ is chosen such that $t \in T_1^{(i)}$ and
consequently $|j/n - t| \leq r_n/n < \delta$.
 Then, 
\begin{align*}
    &\frac{n}{r_n}|R_{n}(t)- \tilde{R}_n(t)| \\
    \leq{}& \frac{n}{r_n} \bigg| \quad \EE\bigg[\exp\bigg(-\sum_{s \in \mathcal{S}} \sum_{j=(i-1)r_n+1}^{ir_n} g(t,s,a_{n}^{-1}\bm X_{j}) \one\{r^{(s)}(\bm X_j) > a_n u\}\bigg)\bigg] \\ 
    &\qquad -\EE\bigg[\exp\bigg(-\sum_{s \in \mathcal{S}}\sum_{j=(i-1)r_n+1}^{ir_n} g(j/n,s,a_{n}^{-1}\bm X_{j}) \one\{r^{(s)}(\bm X_j) > a_n u\} \bigg)\bigg]\bigg| \displaybreak[0]\\
    \leq{}& \frac{n}{r_n} \bigg( \quad \EE\bigg[\exp\bigg(-\sum_{s \in \mathcal{S}} \sum_{j=(i-1)r_n+1}^{ir_n} (g(t,s,a_{n}^{-1}\bm X_{j})-\varepsilon)_+ \one\{r^{(s)}(\bm X_j) > a_n u\}\bigg)\bigg] \\ 
    &\qquad -\EE\bigg[\exp\bigg(-\sum_{s \in \mathcal{S}}\sum_{j=(i-1)r_n+1}^{ir_n} [(g(t,s,a_{n}^{-1}\bm X_{j}) - \varepsilon)_+ + 2\varepsilon] \one\{r^{(s)}(\bm X_j) > a_n u\} \bigg)\bigg]\bigg).
\end{align*}
Applying Lemma \ref{Theorem_sim_cond_4.3} with $v=1$ to the functions
$$ g_1(t,s,\bm x) = (g(t,s,\bm x)-\varepsilon)_+ \quad \text{and} \quad
   g_2(t,s,\bm x) = g_1(t,s,\bm x) + 2\varepsilon, $$
where $g_2$ is appropriately modified outside $[0,1] \times \mathcal{X}_{u}$ to become compactly supported while remaining non-negative and continuous, yields that
\begin{align} \label{eq:upper-bound-rndiff}
    &\lim_{n \to \infty} \frac{n}{r_n}|R_{n}(t)- \tilde{R}_n(t)| \nonumber \\
    \leq{}& \lim_{n \to \infty} \frac{n}{r_n} \PP(M_{r_n}^X > a_n) \nonumber \\
    & \cdot \bigg\{  \EE\bigg[\exp\bigg(-\sum_{s \in \mathcal{S}} \sum_{j=(i-1)r_n+1}^{ir_n} \thresh{g_1}_u(t,s,a_{n}^{-1} \bm X_{j}) \bigg)\, \bigg| \, M_{r_n}^X > a_n\bigg]  \nonumber \\
    & \quad - 
    \EE\bigg[\exp\bigg(-\sum_{s \in \mathcal{S}} \sum_{j=(i-1)r_n+1}^{ir_n} \thresh{g_2}_u(t,s,a_{n}^{-1} \bm X_{j}) \bigg)\, \bigg | \, M_{r_n}^X > a_n\bigg] \bigg\} \displaybreak[0] \nonumber \\ 
    ={}&\lim_{n \to \infty} \frac{n}{r_n} \PP(M_{r_n}^X > a_n) \nonumber \\
    & \cdot \bigg(-\theta^{-1} \EE\bigg[\exp\bigg \{-\sum_{s \in \mathcal{S}} \sum_{j=1}^{\infty} \thresh{g_1}_u(t,s,\bm Y_{j}) \bigg\} - \exp\bigg\{-\sum_{s \in \mathcal{S}} \sum_{j=0}^{\infty} \thresh{g_1}_u(t,s,\bm Y_{j}) \bigg\}\bigg] \nonumber \\
    & \qquad + \theta^{-1} \EE\bigg[\exp\bigg(-\sum_{s \in \mathcal{S}} \sum_{j=1}^{\infty} (\thresh{g_1}_u(t,s,\bm Y_{j}) 
    + 2 \varepsilon \one\{r^{(s)}(\bm Y_{j}) > u\})\bigg) \nonumber \\
    & \qquad \qquad \qquad
    - \exp\bigg(-\sum_{s \in \mathcal{S}} \sum_{j=0}^{\infty} (\thresh{g_1}_u(t,s,\bm Y_{j}) + 2 \varepsilon \one\{r^{(s)}(\bm Y_{j}) > u\})\bigg) \bigg] \bigg) \nonumber \\
    ={}& \lim_{n \to \infty} \frac{n}{r_n} \PP(M_{r_n}^X > a_n) 
    \cdot \theta^{-1} \bigg(\EE\bigg[\exp\bigg(-\sum_{s \in \mathcal{S}} \sum_{j=1}^{\infty} \thresh{g_1}_u(t,s,\bm Y_{j})\bigg)\bigg]  \cdot   \nonumber\\
    &\bigg( \EE\bigg[\exp\bigg(-\sum_{s \in \mathcal{S}} \sum_{j=1}^{\infty} 2 \varepsilon \one\{r^{(s)}(
    \bm Y_j) > u\}\bigg)\bigg]-1\bigg)  \nonumber\\
    & - \EE\bigg[\exp\bigg(-\sum_{s \in \mathcal{S}} \sum_{j=0}^{\infty} \thresh{g_1}_u(t,s,\bm Y_{j})\bigg)\bigg] \cdot \nonumber\\
    & \bigg(\EE\bigg[\exp\bigg(-\sum_{s \in \mathcal{S}} \sum_{j=0}^{\infty} 2 \varepsilon \one\{r^{(s)}( \bm Y_j) > u \}\bigg)\bigg]-1\bigg)  \bigg)  \nonumber\\
    ={}& \lim_{n \to \infty} \frac{n}{r_n} \PP(M_{r_n}^X > a_n) 
    \cdot \theta^{-1} \cdot   \nonumber\\
    &\bigg(\EE\bigg[\exp\bigg(-\sum_{s \in \mathcal{S}} \sum_{j=0}^{\infty} \thresh{g_1}_u(t,s,\bm Y_{j})\bigg)\bigg] \bigg( 1-\EE\bigg[\exp\bigg(-\sum_{s \in \mathcal{S}} \sum_{j=0}^{\infty} 2 \varepsilon \one\bigg\{r^{(s)}( \bm Y_j) > u \bigg\}\bigg)\bigg]\bigg)  \nonumber\\
    & - \EE\bigg[\exp\bigg(-\sum_{s \in \mathcal{S}} \sum_{j=1}^{\infty} \thresh{g_1}_u(t,s,\bm Y_{j})\bigg)\bigg] \bigg(1-\EE\bigg[\exp\bigg(-\sum_{s \in \mathcal{S}} \sum_{j=1}^{\infty} 2 \varepsilon \one\bigg\{r^{(s)}( \bm Y_j) > u \bigg\}\bigg)\bigg]\bigg)  \bigg) \displaybreak[0] \nonumber\\
    \leq{}& \lim_{n \to \infty} \frac{n}{r_n} \PP(M_{r_n}^X > a_n) 
    \cdot \theta^{-1} \cdot \EE\bigg[\exp\bigg(-\sum_{s \in \mathcal{S}} \sum_{j=0}^{\infty} \thresh{g_1}_u(t,s,\bm Y_{j})\bigg)\bigg]\cdot  \nonumber \\
    & \qquad \bigg(1-\EE\bigg[\exp\bigg(-2\varepsilon\sum_{s \in \mathcal{S}} \sum_{j=0}^{\infty} \one\bigg\{r^{(s)}(\bm Y_j) > u \bigg\}\bigg)\bigg] \bigg) \displaybreak[0] \nonumber \\
    \leq{}& \EE\bigg[\exp\bigg(-\sum_{s \in \mathcal{S}} \sum_{j=0}^{\infty} \thresh{g_1}_u(t,s,\bm Y_{j})\bigg)\bigg] \cdot \bigg(
    1-\EE\bigg[\exp\bigg(-2\varepsilon|\mathcal{S}| \sum_{j=0}^{\infty} \one\bigg\{\|\bm Y_j\| > 1 \bigg\}\bigg)\bigg]\bigg) \nonumber \\
\end{align}
where we used the identity $\thresh{g_2}_u(t,s,\bm x) = \thresh{g_1}_u(t,s,\bm x) + 2 \varepsilon \one\{r^{(s)}(\bm x) > u\}$ and the fact that $\lim_{n \to \infty} \frac{n}{r_n} \PP(M_{r_n}^X > a_n) = \theta$
by Proposition 4.2 in \citet{basrak2009regularly}.
By the same proposition,
$$ 1=  \PP\Big(\lim_{j \to \infty} \|\bm Y_j\|=0\Big) \leq  \PP\bigg(\sum_{j=0}^{\infty} \one\bigg\{\|\bm Y_j\| > 1 \bigg\}<\infty\bigg),$$
using the fact that $\lim_{j \to \infty}\|\bm Y_j\|=0$ implies
 $  \|\bm Y_j\| > 1$ for only finitely many $j$.
Consequently, the right-hand side of \eqref{eq:upper-bound-rndiff} becomes arbitrarily small as $\varepsilon \to 0$, which allows us to conclude that
\begin{align} \label{difference_R_n_and_tilde_R_n}
    \lim_{n \to \infty} \frac{n}{r_n}|R_{n}(t)- \tilde{R}_n(t)| = 0 \quad \text{for all } t \in [0,1].
\end{align}
Hence, using the dominated convergence theorem with majorant 
$\sup_{n \in \NN} \frac{n}{r_n} \PP(M_{r_n}^{X}>a_{n})$
and \eqref{difference_R_n_and_tilde_R_n}, we finally obtain
\begin{align*}
& \lim_{n \to \infty} \EE \bigg[ \exp\bigg\{ -\sum_{s \in \mathcal{S}} \sum_{j=1}^{n} \thresh{g}_u(j/n,s,a_{n}^{-1} \bm X_{j}) \bigg\} \bigg] \\    
&=\lim_{n \to \infty} \exp\bigg(-\frac{n}{r_n}\int_{0}^{1} R_n(t) \,\mathrm{d}t \bigg)
= \lim_{n \to \infty} \exp\bigg(-\frac{n}{r_n}\int_{0}^{1} \tilde{R}_n(t) \,\mathrm{d}t \bigg).
\end{align*}
Then, 
\begin{align*}
 &\frac{n}{r_n} \tilde{R}_n(t)=  \frac{n}{r_n} \bigg(1-\EE\bigg[\exp\bigg(-\sum_{s \in \mathcal{S}}\sum_{j=1}^{r_n} \thresh{g}_u(t,s,a_{n}^{-1} \bm X_{j})\bigg)\bigg] \bigg) \\
 &= \frac{n}{r_n} \EE\bigg[\bigg(1-\exp\bigg(-\sum_{s \in \mathcal{S}}\sum_{j=1}^{r_n} \thresh{g}_u(t,s,a_{n}^{-1} \bm X_{j})\bigg)\bigg) \mathds{1}\{M_{r_n}^{X} > a_{n}\}\bigg]  \\
 &=  \frac{n}{r_n} \PP(M_{r_n}^{X}>a_{n})
 \bigg(1-\EE\bigg[\exp\bigg(-\sum_{s \in \mathcal{S}}\sum_{j=1}^{r_n} \thresh{g}_u(t,s,a_{n}^{-1} \bm X_{j})\bigg)\, \bigg| \, M_{r_n}^{X}>a_{n}\bigg]\bigg). 
\end{align*}
Let 
\begin{align}
\label{g_t_definition}
g_{t}((\bm y_j)_{j \in \NN_0})
={}&  \exp\bigg \{-\sum_{s \in \mathcal{S}} \sum_{j=1}^{\infty} \thresh{g}_u(t,s,\bm y_{j}) \bigg\}- \exp\bigg\{-\sum_{s \in \mathcal{S}} \sum_{j=0}^{\infty} \thresh{g}_u(t,s,\bm y_{j}) \bigg\}.
\end{align}
Thus, it holds with Equation \eqref{theta_limes_finite} and Lemma \ref{Theorem_sim_cond_4.3} with $v=1$ that
\begin{align*}
&\lim_{n \to \infty}\frac{n}{r_n} \tilde{R}_n(t)
= \theta \left(1-(1-\theta^{-1} \EE\big[g_{t}((\bm Y_j)_{j \in \NN_0}) \big] \right) = \EE\big[g_{t}((\bm Y_j)_{j \in \NN_0}) \big].
\end{align*}
Thus, we obtain
\begin{align*} 
	& \lim_{n \to \infty} \EE \bigg[ \exp\bigg\{ -\sum_{s \in \mathcal{S}} \sum_{j=1}^{n} \thresh{g}_u(j/n,s,a_{n}^{-1} \bm X_{j}) \bigg\} \bigg] \\
	&= \lim_{n \to \infty} \exp\bigg(-\frac{n}{r_n}\int_{0}^{1} \tilde{R}_n(t) \,\mathrm{d}t \bigg) \displaybreak[0] \\
	&= \exp\bigg(-\int_{0}^{1}  \EE\bigg[g_{t}(( \bm Y_t)_{t \in \NN_0}) \bigg] \,\mathrm{d}t \bigg) \\ 
	&=  \exp\bigg(-\int_{0}^{1}  \EE\bigg[\bigg(\exp\bigg(-\sum_{s \in \mathcal{S}} \sum_{j=1}^{\infty}\thresh{g}_u(t,s,\bm Y_{j})\bigg) \\
	& \qquad \qquad \qquad \qquad \quad -\exp\bigg(-\sum_{s \in \mathcal{S}} \sum_{j=0}^{\infty}\thresh{g}_u(t,s,\bm Y_{j}) \bigg) \bigg)\bigg] \,\mathrm{d}t \bigg).
\end{align*}

Thus, we have proven that $N_n(u)$ converges weakly to a point process $N(u)$ whose Laplace functional is given by \eqref{eq:laplace-tailprocess}. In order to prove that \eqref{eq:laplace-tailprocess} is equal to \eqref{eq:laplace-spectraltailprocess}, we now proceed similarly to the proof of Lemma \ref{Theorem_sim_cond_4.3} and approximate the indicator functions $\one_{(1,\infty]}$ by continuous functions $f_n: \RR \to [0,1]$ such that
$f_n \to \one_{(1,\infty)}$ pointwise as $n \to \infty$.
By applying the bounded convergence theorem to the continuous and uniformly bounded functions that map $(\bm y_t)_{t \in \NN_0}$ to 
\begin{align*}
& \exp\bigg \{-\sum_{s \in \mathcal{S}} \sum_{j=1}^{\infty} g(t,s, \bm y_{j})f_n(r^{(s)}( \bm y_j)) \bigg\} \\
 &- \exp\bigg \{-\sum_{s \in \mathcal{S}} \sum_{j=0}^{\infty} g(t,s, \bm y_{j})f_n(r^{(s)}( \bm y_j)) \bigg\},
\end{align*}
we see that the convergence in Theorem 3.1 (ii) in \citet{basrak2009regularly} is also true for $g_t$, that is, 
\begin{align}
\label{Theorem_31_gt}
&\EE\bigg[g_{t}((\bm Y_j)_{j \in \NN_0}) \bigg]
= \int_{1}^{\infty}\EE\left[ g_{t}((v \bm \Theta_j)_{j \in \NN_0})\right]\alpha v^{-\alpha-1} \,\mathrm{d}v \nonumber \\
&{}=\int_{0}^{\infty}\EE\left[ g_{t}((v \bm \Theta_j)_{j \in \NN_0})\right]\alpha v^{-\alpha-1} \,\mathrm{d}v, 
\end{align}
where the last equality holds because $\thresh{g}_u(t,s,v \bm \Theta_{0})=0$ if $$r^{(s)}(v \bm \Theta_{0}) \leq u v \|\bm \Theta_{0}\| \leq u,$$
which is true for $v \in (0,1]$. 
Consequently, it is
\begin{align*} 
& \lim_{n \to \infty} \EE \bigg[ \exp\bigg\{ -\sum_{s \in \mathcal{S}} \sum_{j=1}^{n} \thresh{g}_u(j/n,s,a_{n}^{-1} \bm X_{j}) \bigg\} \bigg] \\
&= \exp\bigg(-\int_{0}^{1}  \EE\bigg[g_{t}(( \bm Y_j)_{j \in \NN_0}) \bigg] \,\mathrm{d}t \bigg) \\
&= \exp\bigg(-\int_{0}^{1} \int_{1}^{\infty}  \EE\left[ g_{t}((v \bm \Theta_j)_{j \in \NN_0})\right]\alpha v^{-\alpha-1} \,\mathrm{d}v   \,\mathrm{d}t \bigg) \\
&= \exp\bigg(-\int_{0}^{1} \int_{0}^{\infty}\EE\left[g_{t}((v \bm \Theta_j)_{j \in \NN_0})\right]\alpha v^{-\alpha-1} \,\mathrm{d}v   \,\mathrm{d}t \bigg) \\
&= \exp\bigg(-\int_{0}^{1} \int_{0}^{\infty}\EE\bigg[      
\exp\bigg \{-\sum_{s \in \mathcal{S}} \sum_{j=1}^{\infty} \thresh{g}_u(t,s,v \bm \Theta_{j}) \bigg\} \\
&\hspace*{40mm}- \exp\bigg\{-\sum_{s \in \mathcal{S}} \sum_{j=0}^{\infty} \thresh{g}_u(t,s,v \bm \Theta_{j}) \bigg\} \bigg]\alpha v^{-\alpha-1} \,\mathrm{d}v   \,\mathrm{d}t \bigg).
\end{align*}
\end{proof}

The two lemmas above now enable us to prove the two main theorems of the paper, i.e., Theorem~\ref{point_process_N_explicit_representation} and Theorem~\ref{thm_laplacefunctional_marked}.

\begin{proof}[Proof of Theorem~\ref{point_process_N_explicit_representation}]
Let $\widetilde{N}(u)$ denote the point process on the right-hand side of Equation~\eqref{eq:cluster-sum}. Then,
\begin{align} \label{eq:pgf-poisson}
 \Psi_{\tilde{N}(u)}(g) ={}& \EE\bigg[ \exp\bigg( - \sum_{i=1}^{T} \sum_{s \in \mathcal{S}}\sum_{j=0}^{\infty} g(V_i,s, \bm Z_{ij}) \mathds{1}\{r^{(s)}(Z_{ij}) > u\} \bigg)\bigg] \nonumber \\ 
 ={}&\EE\bigg[ \exp\bigg( - \sum_{i=1}^{T} \sum_{s \in \mathcal{S}} \sum_{j=0}^{\infty} \thresh{g}_u(V_i,s,\bm Z_{ij}) \bigg)\bigg] \nonumber \\
  ={}& \EE \bigg[\EE\bigg[\prod_{i=1}^{T} \exp\bigg( -\sum_{s \in \mathcal{S}}\sum_{j=0}^{\infty} \thresh{g}_u(V_i,s,\bm Z_{ij})  \bigg) \, \bigg| \, T\bigg] \bigg] \nonumber \\
   ={}&\EE\bigg[ \bigg( \EE\bigg[\exp\bigg( - \sum_{s \in \mathcal{S}} \sum_{j=0}^{\infty} \thresh{g}_u(V,s,\bm Z_{j})\bigg)\bigg] \bigg)^{T}\bigg], 
 \end{align}
 where we used in the last step that $\sum_{s \in \mathcal{S}}\sum_{j=1}^{\infty}\delta_{(V_i,s,\bm Z_{ij})}$, $i \in \NN$, are i.i.d.~copies of the process $\sum_{s \in \mathcal{S}}\sum_{j=1}^{\infty}\delta_{(V,s,\bm Z_{j})}.$
Then, we see that the right-hand side of \eqref{eq:pgf-poisson} is just the probability-generating function of the Poisson random variable $T$, that is $m_{T}(q)=\EE\left[ q^{T}\right]$ with $$q=\EE\bigg[\exp\bigg( - \sum_{s \in \mathcal{S}} \sum_{j=0}^{\infty} \thresh{g}_u(V,s,\bm Z_{j}) \bigg)\bigg],$$ 
which is given by $m_{T}(q)=\EE\left[ q^{T}\right]=\exp(-\theta(1-q))$. Therefore, we get that
\begin{align*}
 & \EE\bigg[ \exp\bigg( - \sum_{s \in \mathcal{S}} \sum_{i=1}^{T}\sum_{j=0}^{\infty} g(V_i,s, \bm Z_{ij}) \mathds{1}\{r^{(s)}(Z_{ij}) > u\} \bigg)\bigg] \\ 
 ={}&\EE\bigg[ \bigg( \EE\bigg[\exp\bigg( - \sum_{s \in \mathcal{S}} \sum_{j=0}^{\infty} \thresh{g}_u(V,s,\bm Z_{j})\bigg)\bigg] \bigg)^{T}\bigg] \\
  ={}&\exp\bigg(-\theta \bigg( 1-  \EE \bigg[\exp\bigg\{ - \sum_{s \in \mathcal{S}} \sum_{j=0}^{\infty}\thresh{g}_u(V,s,\bm Z_{j})  \bigg\}\bigg]\bigg) \bigg) \\
 ={}&\exp\bigg(-\theta \bigg( 1-  \EE \bigg[\int_0^1\exp\bigg\{ -  \sum_{s \in \mathcal{S}} \sum_{j=0}^{\infty}\thresh{g}_u(t,s,\bm Z_{j})  \bigg\}\, \mathrm{d}t\bigg]\bigg) \bigg) \\
 ={}& \exp\bigg(-\theta \int_{0}^{1}\bigg( 1-  \EE \bigg[\exp\bigg\{-\sum_{s \in \mathcal{S}} \sum_{j=0}^{\infty} \thresh{g}_u(t,s,\bm Z_{j})\bigg\}\bigg]\bigg) \, \mathrm{d}t \bigg).
 \end{align*}

Now, we show that 
\begin{align*}
  &\exp\bigg(- \theta \int_{0}^{1}\bigg( 1-  \EE \bigg[\exp\bigg\{-\sum_{s \in \mathcal{S}} \sum_{j=0}^{\infty} \thresh{g}_u(t,s,\bm Z_{j})\bigg\}\bigg]\bigg) \, \mathrm{d}t \bigg) \\
  ={}&\exp\bigg(-\int_{0}^{1} \EE\bigg[\bigg(\exp\bigg(-\sum_{s \in \mathcal{S}} \sum_{j=1}^{\infty} \thresh{g}_u(t,s,\bm Y_{j}) \bigg)-\exp\bigg(-\sum_{s \in \mathcal{S}} \sum_{j=0}^{\infty} \thresh{g}_u(t,s,\bm Y_{j}) \bigg) \bigg)\bigg] \, \mathrm{d}t \bigg).
\end{align*}

By \eqref{result_before_bounded_convergence_1} and 
Lemma \ref{Theorem_sim_cond_4.3}, it holds with $v=1$ that
\begin{align*}
  &\theta \bigg( 1-  \EE \bigg[\exp\bigg\{-\sum_{s \in \mathcal{S}} \sum_{j=0}^{\infty} \thresh{g}_u(t,s,\bm Z_{j})\bigg\}\bigg]\bigg)\\
 ={}&\theta \bigg(1- \theta^{-1}\EE\bigg[\exp\bigg \{-\sum_{s \in \mathcal{S}} \sum_{j=-\infty}^{\infty} \thresh{g}_u(t,s,\bm Y_{j}) \bigg\} \cdot \one\bigg\{\sup_{t \leq -1} \| \bm Y_t\| \leq 1\bigg\}\bigg] \bigg) \displaybreak[0]\\ 
  ={}& \theta \bigg(1- \lim_{n \to \infty}\EE\bigg[\exp\bigg(-\sum_{s \in \mathcal{S}}\sum_{j=1}^{r_n} \thresh{g}_u(t,s,a_{n}^{-1} \bm X_{j}) \bigg) \, \bigg| \, M_{r_n}^{X}> a_{n}   \bigg] \bigg) \displaybreak[0] \\
  ={}&  \theta \bigg(1-\bigg(1-\frac{1}{\theta}\EE\bigg[\bigg(\exp\bigg(-\sum_{s \in \mathcal{S}} \sum_{j=1}^{\infty} \thresh{g}_u(t,s,\bm Y_{j})\bigg) \\
& \hspace{4cm} - \exp\bigg(-\sum_{s \in \mathcal{S}} \sum_{j=0}^{\infty} \thresh{g}_u(t,s,\bm Y_{j})\bigg) \bigg) \bigg] \bigg) \bigg) \\
={}&  \EE\bigg[\bigg(\exp\bigg(-\sum_{s \in \mathcal{S}} \sum_{j=1}^{\infty} \thresh{g}_u(t,s,\bm Y_{j}) \bigg)-\exp\bigg(-\sum_{s \in \mathcal{S}} \sum_{j=0}^{\infty} \thresh{g}_u(t,s,\bm Y_{j}) \bigg) \bigg)\bigg].
\end{align*}
The assertion follows with Lemma \ref{thm_laplacefunctional_neu}.
\end{proof}

\begin{proof}[Proof of Theorem~\ref{thm_laplacefunctional_marked}]
Define the mapping 
\begin{align*}
    \varphi: [0,1] \times \mathcal{S} \times (E \setminus \{\bm 0\}) \to{}&  [0,1] \times \mathcal{S} \times (E \setminus \{\bm 0\}) \times [0,\infty), \\
    (t,s,\bm x) \mapsto{}& (t,s,\bm x, \ell^{(s)}(\bm x)).
\end{align*}
By Theorem \ref{point_process_N_explicit_representation}, it holds 
that $N_n(u) \to N(u)$ converges weakly as $n \to \infty$. Then, by the continuous mapping theorem (see Theorem 7.1.19 in \citet{kulik2020heavy}), it holds $$M_n(u) = N_n(u) \circ \varphi^{-1} \to N(u) \circ \varphi^{-1} = M(u)$$
converges weakly as $n \to \infty$ because $\varphi$ is continuous.
\end{proof}

\section*{Acknowledgements}
This work has been supported by the integrated project “Climate
Change and Extreme Events - ClimXtreme 2 Module B - Statistics (subproject B3.1)” funded
by the Federal Ministry of Research, Technology and Space (BMFTR) with grant number 01LP2323I, which is gratefully acknowledged.
In addition, we thank participants of the Conference on Extreme Value Analysis (EVA 2025) for helpful comments and discussions, which have led us to use site-dependent risk functionals.

\bibliographystyle{abbrvnat}
\bibliography{refs.bib}

 \end{document}